\documentclass[10pt]{article}
\usepackage{amsmath,amssymb}
\usepackage{xcolor}

\usepackage[latin1]{inputenc}
\usepackage[francais]{babel}
\usepackage{color}
\usepackage{amsfonts,amssymb,amsthm,amsmath}
\usepackage{graphics,graphicx}
\usepackage{aeguill}
\usepackage{url}
\usepackage{dsfont}
\usepackage{ulem}
\usepackage{stmaryrd}

\definecolor{blue}{rgb}{0,0,1}
\newcommand{\blue}{\color{blue}}

\newcommand{\red}{\color{red}}

\definecolor{rose}{rgb}{1,0,0.5}

\definecolor{brique}{rgb}{0.14,0.25,0.14}

\theoremstyle{plain}
\newtheorem{theoreme}{Theorem}[section]
\newtheorem{proposition}[theoreme]{Proposition}
\newtheorem*{proposition*}{Proposition}

\newtheorem{corollaire}[theoreme]{Corollary}
\newtheorem*{corollaire*}{Corollaire}
\newtheorem{lemme}[theoreme]{Lemma}
\newtheorem*{lemme*}{Lemma}

\theoremstyle{definition}
\newtheorem*{definition*}{Definition}

\theoremstyle{remark}
\newtheorem*{remarque*}{Remarque}
\newtheorem{remarque}[theoreme]{Remarks}
\newtheorem{exemple}[theoreme]{Exemple}
\newtheorem*{exemple*}{Exemple}

\newtheorem*{exemples*}{Exemples}

\newcommand{\commm}[1]{{}}

\begin{document}

\begin{center}
{\Large \bf On Asymptotic and Finite-Time Stabilization of Bilinear Systems}
\end{center}

\begin{center}

M. Ouzahra \\

LMASI, ENS, University of  Fez
\end{center}

\vspace{1cm}
\medskip
\begin{center}
  Outline
\end{center}

\vspace{0.5cm}

 $I-$ Weak and Strong Stabilization of Bilinear Systems in Hilbert Spaces.

   \vspace{0.5cm}

$II-$ Exponential Stabilization of Bilinear Systems in Banach Spaces.

\vspace{0.5cm}

 $III-$  Stabilization in Finite-Time of Bilinear Systems.

\medskip

\footnote{
\begin{center}
This course was, for the most part, delivered at the Spring School on Control of Dynamical Systems\\
Corsica, March 16--20, 2026.
\end{center}
}
\newpage

 \section{Introduction}
 Bilinear systems arise naturally in the modeling of many physical and engineering
processes, in which the control acts multiplicatively on the state.\\
This course is devoted to the analysis and design of feedback laws for the
stabilization of bilinear control systems evolving in infinite-dimensional
spaces. The first part of the course focuses on weak and strong stabilization
of bilinear systems in a Hilbert space setting. In this first part, the Hilbertian framework is
adopted for pedagogical reasons, as it allows the main ideas underlying
stabilization strategies to be presented in a clear and intuitive manner.
 The existence of weakly stabilizing feedback laws is investigated under an
approximate observability assumption. The course subsequently addresses strong
and polynomial stabilization. Sufficient conditions ensuring strong stability
are presented, based on appropriate observability estimates, together with
feedback strategies leading to polynomial decay rates of the system
trajectories.

A central part of the course is devoted to the exponential stabilization of
bilinear systems, which will be considered  in the context of Banach spaces. Particular attention is paid to the
additional difficulties arising in this framework due to the lack of a
Hilbertian structure, which implies that the differentiability of the state
norm is no longer guaranteed; an essential property for the derivation of
energy estimates.

Finally, the course introduces the concept of finite-time stabilization of bilinear systems. This part highlights recent results and open problems related to finite-time convergence and situates this type of stabilization within the broader framework of nonlinear and infinite-dimensional control theory.

\newpage

\begin{center}
\huge{Part I}

\end{center}

\begin{center}

{\bf  Weak and Strong Stabilization of Bilinear Systems in Hilbert Spaces}

\end{center}

\vspace{0.5cm}

{\it This first part of the course focuses on weak and strong stabilization
of bilinear systems in a Hilbert space setting. This Hilbertian framework is adopted for pedagogical reasons, as it allows the main ideas underlying stabilization strategies to be presented in a clear and intuitive manner.

\section{Necessary  Conditions for  Stability}

Consider the initial value problem
\begin{equation}\label{Ch1Sf}
    \frac{d}{dt} z(t) =Az(t)+F(t,z(t)),\ z(0)=z_{0},
\end{equation}
where\\
$\star$ $H$ is a Hilbert space,\\
$\star$ $A$  generates  a $C_0-$semigroup $S(t) $ on $H$,\\
$\star$  $ F : \mathbb{R}^+\times H\rightarrow H$ is a given function such that $F(t,0)=0, \; \forall
t\ge 0$ (so that $0$ is an equilibrium point for (\ref{Ch1Sf})).\\

$\bullet$ Recall that (\ref{Ch1Sf}) is weakly (resp. strongly) stable at the origin if there is a unique solution $z$ such that\\
$\star$ $0$ is lyapunov stable (i.e.,  every trajectory starting inside the $\alpha-$ball remains inside
the $\varepsilon-$ball for all future time),\\
$\star$  the solution tends to $0$, as $t\to +\infty$, weakly (resp. strongly).

\vspace{0.25cm}

The purpose of this part is to establish the necessary conditions for weak and strong stabilization.\\
  Define the following sets:

$$  \mathcal{F}=\{y\in H/ F(t,S(t)y) = 0, \forall t\ge 0 \},$$
$$ I_w=\{y\in H/ S(t)y\rightharpoonup 0, \;\mbox{as}\; t\to+\infty\}$$
$$ I_s=\{y\in H/ S(t)y\rightarrow 0, \;\mbox{as}\;
t\to+\infty\}.$$
Notice that
\begin{itemize}

\item[$\star$] The semigroup $S(t)$ is  weakly (resp. strongly) stable iff $I_w=H$ (resp. $I_s=H$).

  \item[$\star$] If the initial state belongs to $\mathcal{F}$, then (in case of uniqueness) the solution of~\eqref{Ch1Sf} coincides with that of its linear part.

  \item[$\star$] For the class of controlled systems considered next, the set $\mathcal{F}$ will be the set of unobservable initial states.
\end{itemize}
The following result provides necessary conditions for the stability of  system (\ref{Ch1Sf}).

\begin{proposition}\label{NC-FTS}

A necessary condition for system~(\ref{Ch1Sf}) to be weakly (resp. strongly) stable is that $\mathcal{F} \subset I_w$ (resp. $\mathcal{F} \subset I_s$).\\
\end{proposition}

{\red Proof.}\\
Assume that system~(\ref{Ch1Sf}) is weakly stable, and let $y \in H$ be such that
$F(t, S(t)y) = 0$ for all $t \geq 0$. Then $z(t) = S(t)y$ solves the system~(\ref{Ch1Sf}) with initial state $z(0) = y$. Hence,
$$
S(t)y = z(t) \rightharpoonup 0 \quad \text{as } t \to +\infty.
$$
The case of strong stability follows from similar arguments.

\section{Weak Stabilization}

Let us consider the following bilinear system:
$$
(\text{BLS}) \quad \dot{z}(t)=Az(t)+{\red v(t)Bz(t)},\qquad z(0)=z_0\in H,
$$
with
\\
$\star$ $H$: Hilbert space endowed with the inner product $\langle \cdot, \cdot \rangle$ and the corresponding norm $\|\cdot\|.$\\
$\star$  $A$: generator of a $C_0-$semigroup $S(t)$.\\
$\star$  $B$: bounded linear operator.\\
$\star$  $v(t)$: scalar multiplicative control.

\vspace{0.25cm}

{\it $\bullet$ Numerous real-world problems can be modeled by system (\text{BLS}); these include applications in engineering, nuclear, thermal, chemical, and social processes, among others. Moreover, in the approximation of nonlinear systems, the bilinear model may provides a better approximation than the linear one.

}

\vspace{0.25cm}

$\bullet$  {\blue See} [Mohler, 1973],    [Mohler and Kolodziej, 1980], [Kalouptsidis and  Tsinias, 1984], [Hager, and Pardalos, 1998], [Khapalov, 2010], [Knopov et al., 2000], [Pardalos and Yatsenko, 2009], [Resende Pereira et al., 2021],{\bf $\cdot \cdot$}

\subsection{Necessary condition for weak stabilization}

 The conventional feedback control for the weak stabilization of bilinear systems is typically a quadratic function of the state:
$$v= -\langle z(t),Bz(t)\rangle.$$
 This type of control was considered  by Ball and Slemrod (1979) to establish weak stabilization of bilinear/affine-control systems in the context of a Hilbert space.\\
  We will start with the following  control law:
$$
v:=v_0 = -\langle z(t),Bz(t)\rangle,
$$
which makes the energy $
E(t)=\tfrac12\|z(t)\|^2
$
decreasing.\\
In the sequel, we apply Proposition~\ref{NC-FTS} to system~(\text{BLS}) closed with $v_0$. To this end, let us introduce the following set:
$$
\mathcal{B} = \{\, y \in H \mid B S(t)y = 0,\; \forall t \geq 0 \,\},
$$
and, for a feedback $v(t) = V(z(t))$ with $V : H \to H$, we set
$$
\mathcal{V} = \{\, y \in H \mid V(S(t)y) = 0,\; \forall t \geq 0 \,\}.
$$

The following result provides  necessary conditions for weak stabilization.

\begin{corollaire}

    (i) If the system (\text{BLS}) is weakly  stabilizable, then we have $\mathcal{B} \subset I_w$.

    (ii) A necessary condition for a feedback $v(t)=V(z(t))$ to be a weakly  stabilizing control for the system (\text{BLS}) is $\mathcal{V} \subset I_w$.

\end{corollaire}

  {\red Proof.}

Let $v=V(z)$ be a stabilizing feedback, and let us  take $F(t, y) =  V(y) By$ for all $t \ge 0$, $y \in H$.

Then, in both cases (i)-(ii), the result follows from Proposition~\ref{NC-FTS} and the fact that $\mathcal{B} \subset \mathcal{F}$ and $\mathcal{V} \subset \mathcal{F}$.\\

As an application to  the control $v_0$, we have the following
corollary:\\

\begin{corollaire}
   A necessary condition for the control $v_0$  to be a
 weakly  stabilizing one for
the system (\text{BLS}) is $ G \subset I_w,$  where
$$G=\{y\in H/ \langle BS(t)y,S(t)y\rangle = 0, \forall
t\ge 0 \}.$$
In other words,  for all $y\in H$, we have the implication:

$$
 \langle BS(t)y,S(t)y\rangle = 0, \forall
t\ge 0 \Rightarrow S(t)y\rightharpoonup 0, \;\mbox{as}\; t\to+\infty.
$$
\end{corollaire}

\subsection{Sufficient conditions for weak stabilization}

Based on the properties of the orbits and the $\omega-$limit sets,  a weak stabilization result was  proved in  [Ball and Slemrod, 1979], using the conventional control  $v_0$ under the assumption
\begin{equation}\label{Ch1obsw}
\langle BS(t)y,S(t)y\rangle = 0, \; \forall t\ge 0 \Longrightarrow y
= 0\cdot
\end{equation}

The proof relies on properties of the orbits and the $\omega-$limit sets. Here, we give a direct proof.

The  result concerning weak stability of (\text{BLS}) is as
follows.

\begin{theoreme}
   (Ball and Slemrod, 1979).

Suppose that

 (i) $A$ generates a semigroup of contractions $S(t)$,

(ii) $B\in  \mathcal{L}(H)$ is a compact operator such that (\ref{Ch1obsw}) holds.

Then the control $v_0$ weakly stabilizes (\text{BLS}).

\end{theoreme}

{\red Proof.}

Since the closed-loop operator
$$
F_0: y \mapsto \langle y, By \rangle By
$$
is locally Lipschitz, there exists a unique local mild solution, which is a classical one for $y_0\in D(A)$. In this case we have, by dissipativity of $A$:
$$
\frac{d}{dt} \|y(t)\|^2\le -2 |\langle y, By \rangle|^2,
$$
from which, we can deduce by density that for all $y_0\in H,$
$$
\int_0^t |\langle y(s), By(s) \rangle|^2 ds \le \frac{\|z_0\|}{2}
\, \, \mbox{and}\, \, \|z(t)\| \le \|z_0\|.
$$
In particular,  the solution is global.

 {\bf Step 1}: Let us prove the following  estimate:

$$
\big(\displaystyle\int_{0}^{T}|\langle
BS(s)z(t),S(s)z(t)\rangle|ds\big)^{2}
=O\big( \displaystyle\int_{t}^{t+T}
\displaystyle  |\langle
z(s),Bz(s)\rangle|^{2}   ds\big),
\; t\to+\infty\cdot
$$

The mild solution is given by  the variation of constants formula (VCF)

$$    z(t)=S(t)z_{0} -\displaystyle\int_{0}^{t}\displaystyle \langle z(s),Bz(s)\rangle S(t-s)Bz(s)ds\cdot
$$

Then using  that $\|z(t)\|\leq
\|z_{0}\|, \; \forall t\ge0, $ we get, for any given $T>0$

$$
\|z(t)-S(t)z_{0}\|\leq
\|B\|\|z_{0}\| \bigg(T\displaystyle\int_{0}^{T}
\displaystyle |\langle
z(s),Bz(s)\rangle|^{2} ds\bigg)^{\frac{1}{2}}, \
\forall t\in [0,T].
$$
Using  the fact that $S(t)$ is a semigroup of contractions and   that $\|z(t)\|\leq
\|z_{0}\|, \; \forall t\ge0,$ we
obtain the inequality
$$
|\langle BS(s)z_{0},S(s)z_{0}\rangle|\leq
2\|B\|\|z(s)-S(s)z_{0}\|\|z_{0}\|+|\langle Bz(s),z(s)\rangle|.
$$
Then, it follows from the two above inequalities that
$$
|\langle BS(s)z_{0},S(s)z_{0}\rangle|\leq
2\|B\|^2\|z_{0}\|^2 \bigg(T\displaystyle\int_{0}^{T}
\displaystyle |\langle
z(s),Bz(s)\rangle|^{2} ds\bigg)^{\frac{1}{2}}
+|\langle Bz(s),z(s)\rangle|.
$$
Replacing $z_{0}$ by $z(t)$ and using the superposition property of the solution, we get
$$|\langle BS(s)z(t),S(s)z(t)\rangle|
\leq
2\|B\|^{2}\|z(t)\|^2\big(T\displaystyle\int_{0}^{T}
\displaystyle |\langle
z(t+\tau),Bz(t+\tau)\rangle|^{2} d\tau
\big)^{\frac{1}{2}}$$
$$+|\langle Bz(s+t),z(s+t)\rangle|, \ \forall t,s\geq0.$$
Integrating this last inequality over the interval $[0,T]$ and using
the semigroup property of the solution $z(t)$, it comes
$$\displaystyle\int_{0}^{T}|\langle BS(s)z(t),S(s)z(t)\rangle|ds\leq$$
$$
(2T^{\frac{3}{2}}\|B\|^{2}\|z(t)\|^2 +T^{\frac{1}{2}} ) \big (\displaystyle\int_{t}^{t+T}
\displaystyle |\langle
z(s),Bz(s)\rangle|^{2}  ds\big)^{\frac{1}{2}}\cdot$$
Finally using again the fact that $\|z(t)\|\leq
\|z_{0}\|, \; \forall t\ge0, $ we obtain
$$
  \displaystyle\int_{0}^{T}|\langle BS(s)z(t),S(s)z(t)\rangle|ds  \leq
    \big(2T^{\frac{3}{2}}\|B\|^{2}\|z_0\|^{2}+T^{\frac{1}{2}}\big)  \big(\displaystyle\int_{t}^{t+T}
\displaystyle |\langle
z(s),Bz(s)\rangle|^{2}  ds\big)^{\frac{1}{2}}\cdot
 $$
This gives the estimate
$$
\big(\displaystyle\int_{0}^{T}|\langle
BS(s)z(t),S(s)z(t)\rangle|ds\big)^{2}
=O\big( \displaystyle\int_{t}^{t+T}
\displaystyle  |\langle
z(s),Bz(s)\rangle|^{2}   ds\big),
\; t\to+\infty\cdot
$$

  {\bf Step 2}: Weak convergence to the equilibrium.

Let $\varphi \in H$. We shall show that $\langle z(t), \varphi \rangle \to 0$ as $t \to +\infty$. Let $(t_n) \ge 0$ be a sequence of real numbers such that $t_n \to +\infty$ as $n \to +\infty$.

Since $\|z(t_n)\| \le \|z_0\|$ for all $n \ge 1$, the sequence $x_n := \langle z(t_n), \varphi \rangle$ is bounded, so we can extract a convergent subsequence.

Let $x_{\gamma(n)}$ be an arbitrary convergent subsequence of $x_n$. Since $z(t_{\gamma(n)})$ is bounded in the Hilbert space $H$, we can extract an other subsequence, still denoted by $z(t_{\gamma(n)})$, such that
$$
z(t_{\gamma(n)}) \rightharpoonup \phi \in H \quad \text{as } n \to +\infty.
$$
Using the compactness of $B$, we deduce via the dominated convergence theorem that
$$
\displaystyle\int_{0}^{T}|\langle
BS(s)z(t_{\gamma(n)}),S(s)z(t_{\gamma(n)})\rangle|ds \to \displaystyle\int_{0}^{T}|\langle
BS(s)\phi,S(s)\phi \rangle|ds, n\to+\infty.
$$
Moreover, it follows from the estimate
$$ \big(\displaystyle\int_{0}^{T}|\langle
BS(s)z(t),S(s)z(t)\rangle|ds\big)^{2}
=O\big( \displaystyle\int_{t}^{t+T}
\displaystyle  |\langle
z(s),Bz(s)\rangle|^{2}   ds\big),
\; t\to+\infty
$$
that
$$
\displaystyle\int_{0}^{T}|\langle
BS(s)z(t_{\gamma(n)}),S(s)z(t_{\gamma(n)})\rangle|ds \to 0, n\to+\infty.
$$
Hence
$$
\langle B S(t)\phi, S(t)\phi \rangle = 0, \quad \forall t \ge 0,
$$
which, by virtue of~(\ref{Ch1obsw}), implies $\phi = 0$. Hence, $x_n \to 0$ as $n \to +\infty$, and consequently $\langle z(t), \varphi \rangle \to 0$ as $t \to +\infty$. In other words,
$$
z(t) \rightharpoonup 0 \quad \text{as } t \to +\infty.
$$

From the proof of the above theorem, one can see that weak stability still holds if the compactness of the control operator $ B $ is replaced by that of the semigroup $ S(t) $. In the latter case, it also follows from the (VCF) that the stability is, in fact, strong (see also [Bounit and Hammouri, 1999], [Berrahmoune et al., 2001] and [Berrahmoune, 2010]).\\

We have the following similar result regarding strong stability:

\begin{theoreme}
Suppose that (i) $A$ generates a compact  contractive semigroup $S(t)$ and that
(ii) $B\in  \mathcal{L}(H)$ satisfies (\ref{Ch1obsw}).

Then the control $v_0$ strongly stabilizes (\text{BLS}).
\end{theoreme}

{\bf Proof}
    From the same arguments as in the above theorem, we can see that, by using the compactness of $S(t)$ in place of that of $B$, the solution $z(t)$ of the   system (\text{BLS}), closed with $v_0$, tends weakly to $0$ as $t \to +\infty$.\\
Moreover, for any given $T>0$, the variation of constants formula implies
\[
z(t+T)=S(T)z(t)+\int_{t}^{t+T} v_0(s)\, S(t+T-s)B z(s)\, ds.
\]
By again using the compactness of \(S(t)\), we deduce that $S(T)z(t)\rightarrow 0,$ as $t\to+\infty$.
\\
Moreover, we have
$$
\|\int_{t}^{t+T} v_0(s)\, S(t+T-s)B z(s)\, ds\|
\le \|B\|\:\sqrt{T} \:\|z_0\| \bigg(\int_{t}^{t+T} |\langle z(s),B z(s) \rangle|^2 \, ds\bigg)^{1/2} \to 0\; \mbox{as} \; t\to +\infty.
$$
This achieves the proof.

 \begin{remarque}

\begin{enumerate}

 \item  From the proofs, it is not clear whether the (necessary) condition $G \subset I_w$
is also sufficient for weak stabilization in general.

\item  Notice that, in the case of a self-adjoint operator $B$, we have $G=\mathcal{B}$.
However, if for instance, $B$ is skew-adjoint, then $G=\{0\}$, which is not necessarily the case for $\mathcal{B}$.

\item  Observing that $ G $ is invariant under $ S(t) $, we see that if $ G^\perp $ is invariant under $ S^*(t) $ and both $ G $ and $ G^\perp $ are invariant under $ B, $ then one can use the orthogonal decomposition
$
H = G \oplus G^\perp.
$
Indeed, under the (necessary) condition $G \subset I_w, $ the initial system induces a weakly stable dynamics on $ G $, while on $ G^\perp $, the necessary and sufficient conditions for weak stabilization coincide.

\end{enumerate}

\end{remarque}

\subsection{Application to the undamped wave equation}

Let $\Omega=(0,1)$ and  $Q = \Omega \times (0,+\infty)$, and consider the wave equation

\begin{equation}\label{Ch1wave2}
\left\{
\begin{array}{ll}
\displaystyle \frac{\partial^2 z}{\partial t^2} (x,t) = \displaystyle \frac{\partial^2 z}{\partial x^2}  (x,t) + v(t) z(x,t), & \text{in } Q,\\
\\
z(0,t)=z(1,t) = 0, & \text{on } t\in (0,+\infty).
\end{array}
\right.
\end{equation}
Let  $A_1 = \partial_{xx}$  with $D(A_1)= H_0^1(0,1)\cap H^2(0,1).$ The space $V=L^2(\Omega)$ possesses an orthonormal basis $(\varphi_j=\sqrt 2\sin(j\pi x))_{j \ge 1}$ of eigenfunctions of the operator $-A_1 $   with DBC, associated with eigenvalues $(\lambda_j= (j\pi)^2)_{j \ge 1}$.\\
The system~(\ref{Ch1wave2}) can be written in the abstract form~(BLS) by setting
$$
H = V_1 \times V, \quad
B = \begin{pmatrix} 0 & 0 \\ I & 0 \end{pmatrix} \quad \, \mbox{and} \,
A = \begin{pmatrix} 0 & I \\ A_1 & 0 \end{pmatrix}.
$$
For $z_1 = (y_1, w_1)$ and $z_2 = (y_2, w_2)$ in $H$, we define
$$
\langle z_1, z_2 \rangle = \langle y_1, y_2 \rangle_{H_0^1(0,1)} + \langle w_1, w_2 \rangle_{L^2(0,1)},
$$
so that $H$ is a separable Hilbert space.\\
Moreover, the operator $A$ generates the semigroup $S(t)$ defined, for all $\psi = \sum_{n=1}^\infty \begin{pmatrix} a_n \\ b_n \end{pmatrix} \varphi_n$, by
$$
S(t)\psi = \sum_{n=1}^\infty
\begin{pmatrix}
a_n \cos(n\pi t) + \frac{b_n}{n\pi} \sin(n\pi t) \\
\\
- n\pi a_n \sin(n\pi t) + b_n \cos(n\pi t)
\end{pmatrix} \varphi_n.
$$
Then
$$
\langle B S(t)\psi, S(t)\psi \rangle = \sum_{n=1}^\infty n\pi \Big[ \frac{1}{2} \Big(a_n^2 - \frac{b_n^2}{(n\pi)^2} \Big) \sin(2  n\pi t) - \frac{a_n b_n}{n\pi} \cos(2  n\pi t) \Big].
$$
Hence,
$$
\langle S(t)\psi, B S(t)\psi \rangle = 0, \ \forall t \ge 0 \quad \Rightarrow \quad a_n = b_n = 0, \;\; \forall n \ge 1,
$$
$$
\Rightarrow  \psi = 0.
$$
Thus condition~(\ref{Ch1obsw}) holds.\\
We conclude that the feedback control
$$
v_0(t) = -  \displaystyle \int_\Omega z(x,t) \frac{\partial z}{\partial t} (x,t) \, dx $$
ensures the weak stabilization of system~(\ref{Ch1wave2}).

\begin{remarque}

\begin{enumerate}
\item  The authors in [Ball and Slemrod, 1979] raised the question of whether strong convergence holds for the stabilized state of the bilinear undamped wave equation.

\vspace{0.25cm}

\item Strong stabilization of the undamped  bilinear systems has been first proved in [Muller, 1989].
 More precisely, the author showed that  the stability of the bilinear undamped wave equation, with the quadratic control $v_0$,
is indeed strong. However, no a priori estimates for the rate of energy decay are available. It turns out that,
for every solution, the energy converges to zero as $t \to \infty$, but the decay may be arbitrarily slow.

\item  Under a strong observation assumption, the control $v_0$ yields the estimate (see [Berrahmoune, 1999] and [Ouzahra, 2008])
$$
\|z(t)\| = O\Big(\frac{1}{\sqrt{t}}\Big).
$$
 This  estimate is optimal, as can be seen by taking $A = 0$ and $B = I$.

\item Next, we provide a class of controls, which includes the quadratic one, that enhances the degree of stabilization and allows the achievement of any prescribed polynomial decay rate for the system state.

\end{enumerate}

\end{remarque}

\section{ Strong Stabilization}

\subsection{Necessary  conditions }

In the sequel, we give necessary conditions for strong stabilization of system~(\text{BLS}).
To this end, we use again the following sets:
$$
\mathcal{B} = \{\, y \in H \mid BS(t)y = 0,\; \forall t \geq 0 \,\},
$$
and
$$ I_s=\{y\in H/ S(t)y\rightarrow 0, \;\mbox{as}\;
t\to+\infty\}.$$

Then it comes from Theorem-1 that

\begin{corollaire}

    (i)     If the system (\text{BLS})  is  strongly stabilizable, then we have  $\mathcal{B}\subset I_s$.

(ii) If in addition $A$ generates a semigroup  $S(t)$ of isometries on $H$, then we have the following necessary condition for strong stabilizability of the system (\text{BLS}):
\begin{equation}\label{Ch1BS}
 B S(t)y = 0, \forall t\ge 0 \Rightarrow y = 0.
 \end{equation}
\end{corollaire}

As an application to  the control $v_0$, we have the following
corollary:
\begin{corollaire}
  (i) A necessary condition for the control $v_0$  to be a
  strongly stabilizing one for
the system (\text{BLS})  is $G\subset I_s,$ where $G=\{y\in H/ \langle BS(t)y,S(t)y\rangle = 0, \forall
t\ge 0 \}$.

(ii) If  $A$ generates a semigroup  $S(t)$ of {\blue isometries} on $H$, then
 the following condition:
\begin{equation}\label{Ch1obswbis}
\langle BS(t)y,S(t)y\rangle = 0, \; \forall t\ge 0 \Longrightarrow y
= 0,
\end{equation}
 is necessary for the feedback $v_0$ to be a strongly stabilizing control for (\text{BLS}) .

\end{corollaire}

\begin{remarque}

\begin{enumerate}

\item If $S(t)$ is not a semigroup of isometries, then the condition (\ref{Ch1obswbis})  is not, in general,  necessary for the stabilization
of (\text{BLS})  by the control $v_0$, as we can see when $S(t)$ is stable and $B=0$.

\item The point (ii) of the corollary does not apply for weak stabilizability, as we can see by taking
$Az=- \displaystyle\frac{\partial}{\partial x} z, \forall z\in D(A)=\{z\in H^1(0,\infty)\;/z(0)=0\} $ and $B=0$.

\item  The condition  (\ref{Ch1obsw}) is weaker than the following  one:
\begin{equation}\label{Ch1balT}
\langle B S(t)z, S(t)z\rangle = 0, \quad \forall\, 0 \leq t \leq T \;\;\Rightarrow\;\; z = 0.
\end{equation}

\item Note that $(\ref{Ch1obswbis}) \not\Rightarrow (\ref{Ch1balT})$, as can be seen by taking
$$
H=L^2(0,+\infty), \quad A = -\frac{\partial}{\partial x}, \quad
\mathcal{D}(A) = \{y\in H^1(0,+\infty) \mid y(0)=0\},
$$
and
$$
Bz = \mathbf{1}_{[a,+\infty)} \, z, \quad \mbox{with} \quad a>T.
$$

\end{enumerate}

\end{remarque}

\subsection{Sufficient conditions for strong stabilization}
$\bullet$  To construct a stabilizing control $ v(t) $ for system (BLS), a natural approach is to formally compute the time derivative of the energy
$$
E(t) := \frac{1}{2} \| z(t) \|^2,
$$
yielding:
$$
\frac{dE(t)}{dt} \le {\cal R}e \left( v(t) \langle Bz(t), z(t) \rangle \right).
$$
$\bullet$  To ensure that the energy is non-increasing, one may consider feedback controls of the form $ v(t) = f(z(t)) $, where $ f $ satisfies the following energy dissipation inequality:
$$
{\cal R}e \left( f(z(t)) \langle Bz(t), z(t) \rangle \right) \le 0, \quad \forall t \ge 0.
$$
$\bullet$ As a class of feedback controls that satisfy this inequality, we consider the following family of controls:

\begin{equation}\label{Ch1p}
v_r(t) = -\frac{\langle z(t), Bz(t) \rangle}{\| z(t) \|^r} \mathbf{1}_{\big(z(t) \ne 0 \big)}, \quad r \in \mathbb{R}.
\end{equation}
In the following, we investigate the  strong stabilization of system (BLS) using the control (\ref{Ch1p}).\\
With the control (\ref{Ch1p}), the system (BLS) becomes
\begin{equation}\label{Ch1cloosed}
\frac{dz(t)}{dt}=Az(t)+F_r(z(t)),\ z(0)=z_{0},
\end{equation}
where $F_r(y)=-\displaystyle\frac{\langle y,By\rangle }{\|y\|^{r}}By, $ if $ y\neq0$ and  $F_r(0)=0.$\\

The following result provides conditions for the existence and uniqueness of a solution to the closed-loop system (\ref{Ch1cloosed}).

\begin{theoreme}

Let $A$ generate a semigroup of contractions $S(t)$ and let $B\in
 \mathcal{L}(H)$. Then for all $r\in (-\infty,2]$, the system
(\ref{Ch1cloosed}) possesses a unique global mild solution.

\end{theoreme}

{\red Proof}.

 Let us show that the map $F_r$ is locally Lipschitz
from $H$ to $H$. Let  $z,y\in H$ such that
 $ 0<\|y\| \le \|z\| \le R$.
 Then we have
$$
 \frac{\langle Bz,z\rangle}{\|z\|^r}Bz- \frac{\langle By,y\rangle}{\|y\|^r}By =
\frac{\langle Bz,z \rangle Bz-\langle By,y\rangle
By}{\|z\|^r}-\frac{(\|z\|^r-\|y\|^r)\langle By,y \rangle
By}{\|z\|^r\|y\|^r}\cdot
$$
From the relation
$$
 \langle Bz,z \rangle Bz-\langle By,y\rangle By=\langle Bz,z \rangle (Bz-By)+(\langle Bz,z-y\rangle+\langle Bz-By,y\rangle )By,
$$
we deduce that
$$
\frac{\|\langle Bz,z \rangle Bz-\langle By,y \rangle
By\|}{\|z\|^r}\le 3 \|B\|^2\|z-y\|\|z\|^{2-r}.
$$
It follows that
$$
\begin{array}{ccc}
  \|\frac{\langle Bz,z \rangle}{\|z\|^r}Bz-
\frac{\langle By,y\rangle}{\|y\|^r}By\| &\le& 3
\|B\|^2\|z-y\|\|z\|^{2-r} \\&+&\|B\|^2
|\|z\|^r-\|y\|^r|\frac{\|y\|^{3-r}}{\|z\|^{r}}\cdot
\end{array}
$$
Two situations arise:\\

{\bf Case 1 :} $1\le r \le 2$. \\

By making use of the real function  $ t\mapsto t^r,$ one can see
that
$$ |\|z\|^r-\|y\|^r| \le r \|z\|^{r-1}|\|z\|-\|y\||, $$
 and hence $$ |\|z\|^r-\|y\|^r|\frac{\|y\|^{3-r}}{\|z\|^{r}}\le r \|z-y\|\|y\|^{2-r}\cdot$$
We deduce that $F_r$ is locally Lipschitz.\\

{\bf Case 2 :}  $r < 1$.\\

Remarking that   $F_r=\|y\|^{2-r}F_2$ and
\begin{equation}\label{Ch1Fr}
F_r(y)-F_r(z) =
\|y\|^{2-r}(F_2(y)-F_2(z))+(\|y\|^{2-r}-\|z\|^{2-r})F_2(z),
\end{equation}
we deduce, since $2-r\ge 1$, that $F_r$ is locally
Lipschitz. Then the system (\text{BLS}) admits a unique mild solution defined on a
maximal interval $[0,t_{\max})$.\\
Then, using approximation argument, we can show as in the case $r=0$ that
\begin{equation}\label{Ch1eneg}
\|z(t)\|^{2} -  \|z(\tau)\|^2 \leq
-2\displaystyle\int_\tau^t\frac{|\langle
z(s),Bz(s)\rangle|^{2}}{\|z(s)\|^{r}}  ds,\; 0\le \tau\le t <
t_{\max}\cdot
\end{equation}
It follows that $\|z(t)\|\leq \|z_{0}\|,\ \forall t\in
[0,t_{\max})$, and hence $z(t)$ is a global solution i.e.,
$t_{\max}=+\infty$.\\
The stability result of (\ref{Ch1cloosed}) is stated as follows.

\begin{theoreme}
  Suppose that (i) $A$ generates a semigroup of contractions $S(t)$,
(ii) $B\in  \mathcal{L}(H)$  and (iii) there exist $ \delta, T>0$  such that
\begin{equation}\label{Ch1obss}
\int_{0}^{T}|\langle BS(t)y,S(t)y\rangle|dt\geq \delta\|y\|^{2},\
\forall y \in H.
\end{equation}

 Then for all ${\red r\in (-\infty,2)}$,  the feedback (\ref{Ch1p}) strongly stabilizes (BLS) with
the following decay estimate

\begin{equation}\label{Ch1estim}
{\blue \|z(t)\|=O(t^{-\frac{1}{2-r}})},\ as \  t\rightarrow +\infty.
\end{equation}

\end{theoreme}

\begin{remarque}

\begin{enumerate}

\item By taking $ B = C^*C $ in the estimate (\ref{Ch1obss}), where $ C $ is an output operator, one recovers the classical concept of exact observability for linear systems.
In this case, the quadratic  feedback $v_0$ (i.e., when $r=0$) depends only on the output $ C y(t) $.

\item Notice that the estimate~(\ref{Ch1estim}) is optimal for the control $v_r$.

\end{enumerate}

\end{remarque}

 {\red Proof.}

  First, let us notice that we can suppose that $y(t)\ne 0, \forall t\ge 0,$ since otherwise we will have $y(t)=0, t\ge \tau$ for some $\tau >0.$\\
 Using  the fact that $S(t)$ is a semigroup of contractions, we
obtain the inequality
$$
|\langle BS(s)z_{0},S(s)z_{0}\rangle|\leq
2\|B\|\|z(s)-S(s)z_{0}\|\|z_{0}\|+|\langle Bz(s),z(s)\rangle|.
$$
It follows from the variation of parameters formula  that
   $$ S(t)z_{0} -z(t)=\displaystyle\int_{0}^{t}\displaystyle\frac{\langle z(s),Bz(s)\rangle}{\|z(s)\|^{r}} S(t-s)Bz(s)ds\cdot
$$
Then using Schwartz's inequality and the fact that $\|z(t)\|\leq
\|z_{0}\|, \; \forall t\ge0, $ we get
$$
\|z(t)-S(t)z_{0}\|\leq
\|B\|\|z_{0}\|^{1-\frac{r}{2}}\big(T\displaystyle\int_{0}^{T}
\displaystyle\frac{|\langle
z(s),Bz(s)\rangle|^{2}}{\|z(s)\|^{r} }ds\big)^{\frac{1}{2}}, \
\forall t\in [0,T].
$$
Replacing $z_{0}$ by $z(t)$ and using the fact that
$$
\|z(t)\|\le \|z_0\|,\; \forall t\ge 0,
$$
we get
$$\hspace{-0.2cm}|\langle BS(s)z(t),S(s)z(t)\rangle|
\leq
2\|B\|^{2}\|z(t)\|^{2-\frac{r}{2}}\bigg(T\displaystyle\int_{0}^{T}
\displaystyle\frac{|\langle
z(t+\tau),Bz(t+\tau)\rangle|^{2}}{\|z(t+\tau)\|^{r}} d\tau
\bigg)^{\frac{1}{2}}$$
$$+|\langle Bz(s+t),z(s+t)\rangle|, \ \forall t,s\geq0.$$

$\bullet$ Let us first suppose that $r\in[0,2).$\\

 Then integrating this last inequality over the interval $[0,T]$ and using H\"older's
inequality, it follows  that
$$\displaystyle\int_{0}^{T}|\langle BS(s)z(t),S(s)z(t)\rangle|ds\leq$$
$$
(2T^{\frac{3}{2}}\|B\|^{2}\|z(t)\|^{2-\frac{r}{2}}+T^{\frac{1}{2}}\|z(t)\|^{\frac{r}{2}})(\displaystyle\int_{t}^{t+T}
\displaystyle\frac{|\langle
z(s),Bz(s)\rangle|^{2}}{\|z(s)\|^{r}} ds)^{\frac{1}{2}}\cdot$$
Finally, we obtain (since $0<r<2$)
\begin{equation}\label{Ch1estim-C1}
\begin{array}{ccc}
  \displaystyle\int_{0}^{T}|\langle BS(s)z(t),S(s)z(t)\rangle|ds \leq\\
\big(2T^{\frac{3}{2}}\|B\|^{2}\|z_0\|^{2-r}
+T^{\frac{1}{2}}\big)\|z(t)\|^{\frac{r}{2}} \big(\displaystyle\int_{t}^{t+T}
\displaystyle\frac{|\langle
z(s),Bz(s)\rangle|^{2}}{\|z(s)\|^{r}} ds\big)^{\frac{1}{2}}\cdot
\end{array}
\end{equation}
This gives the estimate
$$
\bigg(\displaystyle\int_{0}^{T}|\langle
BS(s)z(t),S(s)z(t)\rangle|ds\bigg)^{2}
=O\bigg(\|z(t)\|^r\displaystyle\int_{t}^{t+T}
\displaystyle\frac{|\langle
z(s),Bz(s)\rangle|^{2}}{\|z(s)\|^{r}}  ds\bigg),
$$
as $ t\to +\infty\cdot$\\
Moreover, letting $ K_0=2T^{\frac{3}{2}}\|B\|^{2}\|z_0\|^{2-r}+T^{\frac{1}{2}}$, we deduce from the estimate (\ref{Ch1estim-C1}),
 the observation inequality  (\ref{Ch1obss}) and the estimate (\ref{Ch1eneg}) that
\begin{equation}\label{Ch1estim2-C1}
{\blue \delta^2 \|y(t+T)\|^{4-r}\le K_0^2 \big (\|y(t)\|^{2} -\|y(t+T)\|^{2} \big).}
\end{equation}
Let us recall that for all $r<2,$ we have
$$\displaystyle\int_{0}^{T}|\langle BS(s)z(t),S(s)z(t)\rangle|ds
\leq$$
$$
(2T^{\frac{3}{2}}\|B\|^{2}\|z(t)\|^{2-\frac{r}{2}}+T^{\frac{1}{2}}\|z(t)\|^{\frac{r}{2}})(\displaystyle\int_{t}^{t+T}
\displaystyle\frac{|\langle
z(s),Bz(s)\rangle|^{2}}{\|z(s)\|^{r}} ds)^{\frac{1}{2}}\cdot$$

$\bullet$ Now for  $r<0,$ we have
$$\|y(\tau)\|^{\frac{r}{2}}\leq \|y(t+T)\|^{\frac{r}{2}}, \,\ \forall \tau \in [t,t+T],$$ so we obtain   the following inequality:
$$
\begin{array}{ll}
  \displaystyle\int_{0}^{T}|\langle BS(s)y(t),S(s)y(t)\rangle|ds & \leq
   \|y(t+T)\|^{\frac{r}{2}}\big(2T^{\frac{3}{2}}\|B\|^2 \|y_{_0}\|^{2-r}+T^{\frac{1}{2}} \big) \\
   & \times \big(\displaystyle\int_{t}^{t+T}\frac{|\langle y(\tau),By(\tau)\rangle|^2}{\|y(\tau)\|^{r}} d\tau\big)^{\frac{1}{2}},
\end{array}
$$
from which (\ref{Ch1estim2-C1}) follows for $r<0$.\\
Let us consider the sequence $s_{k}=\|z(kT)\|^{2}$,
$k\in \mathbb{N}.$ Integrating the inequality (\ref{Ch1eneg}) over the
interval $[kT,(k+1)T]$ and using (\ref{Ch1estim2-C1}), we deduce that
$$s_{k}-s_{k+1}\geq c_{1}
s_{k+1}^{2-\frac{r}{2}},$$
where $c_{1}=\frac{\delta^2}{K_0^2}.$\\
Setting  $u(s)=c_{1}s^{\gamma}$ with $\gamma = 2-\displaystyle\frac{r}{2};$ the last inequality may be
written as
$$ s_{k}\ge u(s_{k+1})+s_{k+1},\ k\geq0.$$
We then conclude using  the next  lemma and the fact that $\|z(t)\|$
decreases.
\begin{lemme} (see [Ammari \& Tucsnak, 2000]).\\
 Let $(s_k)$ be a sequence of positive real numbers satisfying
$$
s_{k+1}\le s_k  - C s_{k+1}^{\alpha+2},\; \forall k\ge 0,
$$
where $C > 0$ and  $\alpha>-1$ are constants.\\
Then there exists a positive constant $M$ (depending on $\alpha$ and $C$) such that
\begin{equation}\label{sk}
s_k \le \frac{M}{(k+1)^{\frac{1}{\alpha+1}}},\; \forall k\ge 0.
\end{equation}

\end{lemme}

{\red Proof of the lemma.}\\
   Define the sequence
$$
f_k=\frac{M}{(k+1)^{\frac{1}{1+\alpha}}},
$$
where $M>0$ will be chosen later.\\
 We will prove that
 $$s_k\le f_k, \; k\ge 0.$$
  Using a first-order expansion for large $k$, we can see that
$$
f_k-f_{k+1}
=
M\Big((k+1)^{-\frac1{1+\alpha}}
-(k+2)^{-\frac1{1+\alpha}}\Big).
$$
$$\sim \frac{M}{1+\alpha}k^{-1} (k+2)^{-\frac{1}{1+\alpha}}
$$
Hence,
$$
\lim_{k\to\infty}
\Big((f_k-f_{k+1})k(k+2)^{\frac1{1+\alpha}}\Big)
=
\frac{M}{1+\alpha}.
$$
Thus there exists $k_0>0$ such that
$$
f_k-f_{k+1}
\le
\frac{2M}{(1+\alpha)k(k+2)^{\frac1{1+\alpha}}},
\qquad k\ge k_0.
$$
We have
$$
\frac{1}{k(k+2)^{\frac1{1+\alpha}}}
\le \frac{2}{(k+2)} \frac{1}{(k+2)^{\frac{1}{1+\alpha}}}=
\frac{2}{(k+2)^{\frac{2+\alpha}{1+\alpha}}},
\qquad k\ge2.
$$
Then, observing that
$$
f_{k+1}^{2+\alpha}
=
\frac{M^{2+\alpha}}{(k+2)^{\frac{2+\alpha}{1+\alpha}}},
$$
we can see that
$$
f_k-f_{k+1}
\le
\frac{4}{(1+\alpha)M^{1+\alpha}}
f_{k+1}^{2+\alpha},
\qquad k\ge k_1 \; (k_1:=k_0+2).
$$
 Now, choose $M$ large enough such that
$$
\frac{4}{(1+\alpha)M^{1+\alpha}}<C \;\;\; \mbox{and}\;\;\;
f_{k_1}\ge s_1.
$$
Then
$$
f_k-f_{k+1}\le C f_{k+1}^{2+\alpha}, \qquad k\ge k_1.
$$
  We prove
$$
s_k\le f_k, \qquad k\ge k_1.
$$
$\star$  For $k=k_1$, the property follows directly from the choice of $M$.
\\
 Let  $k\ge k_1$, and assume that $
s_k\le f_k.$
\\
$\star$  By the induction hypothesis and the estimates above, we obtain
$$
s_{k+1}+C s_{k+1}^{2+\alpha}
\le s_k\le f_k\le
f_{k+1}+C f_{k+1}^{2+\alpha}.
$$
Since the function
$$
\phi(x)=x+C x^{2+\alpha}
$$
is increasing on $\mathbb{R}_+$, this implies
$$
s_{k+1}\le f_{k+1}.
$$
Thus the result holds for $k+1$.\\
 By induction,
$$
s_k\le f_k
=
\frac{M}{(k+1)^{\frac1{1+\alpha}}},
\qquad k\ge k_1,
$$
which give the estimate (\ref{sk}).\\
This also covers the finite number of indices $0\le k < k_1$:
$$
s_k \le s_1\le f_{k_1} \le f_k,\; k=1,..,k_1.
$$

\begin{remarque}

\begin{enumerate}

\item If the Hilbert space $H$ is infinite-dimensional and either the operator
$B$ or the semigroup $S(t)$, $t>0$, is compact, then condition~(\ref{Ch1obss})
cannot be satisfied.

\item In the finite-dimensional case (e.g., $X=R^n$), contraction may be assumed to hold with
respect to another inner product
$$
\langle x, y \rangle_{P} := \langle P x, y \rangle,
$$
where the matrix $P = P^{T}>0$ satisfies the LMI: $PA+A^T P\le 0$.\\
{\blue See} [Gutman, 1981], [Jurjevic and  Quinn, 1978], [Ryan and Buckingham, 1983], [Ryan, E.P. (1984)], [Slemrod, 1978], [Quinn, 1980] {\bf $\cdot \cdot \cdot$}

\end{enumerate}
\end{remarque}

\subsection{ Application to the wave equation with damping}

 Let us consider the following $1D-$wave equation:

\begin{equation}\label{Ch4wave1}
\left\{
\begin{array}{ll}
 z_{tt}(x,t) =  z_{xx}(x,t) + v(t)   z_t(x,t), & \text{in } Q, \\
 \\
z(0,t)=z(1,t)=0 , & t\ge0,
\end{array}
\right.
\end{equation}
where  $Q:=\Omega \times (0,+\infty)$ with $ \Omega = (0,1)$.\\
This system can be written in the abstract form of equation~(BLS) by setting
$$
H = H_0^1(\Omega) \times L^2(\Omega), \quad
\langle (y_1, z_1), (y_2, z_2) \rangle = \langle y_1, y_2 \rangle_{H_0^1(\Omega)} + \langle z_1, z_2 \rangle_{L^2(\Omega)},
$$
and
$$
A = \begin{pmatrix} 0 & I \\ A_1 & 0 \end{pmatrix}, \quad
B = \begin{pmatrix} 0 & 0 \\ 0 & I\!d \end{pmatrix},
$$
where $A_1 = \Delta$ is the Laplacian operator with domain $D(A) = (H^2(\Omega) \cap H_0^1(\Omega)) \times H_0^1(\Omega)$.\\
The spectrum of the operator $-\frac{\partial^2}{\partial x^2}$ with Dirichlet boundary conditions is given by the simple eigenvalues $\lambda_j = (j\pi)^2$, corresponding to eigenfunctions $\phi_j(x) = \sqrt{2}\sin(j\pi x)$, $\forall j \in \mathbb{N}^*$.\\
Let $y = (y_1, y_2) \in H$ with $y_1 = \sum_{j=1}^{\infty} \alpha_j \phi_j$ and $y_2 = \sum_{j=1}^{\infty} \lambda_j^{\frac{1}{2}} \beta_j \phi_j$, where $(\alpha_j, \beta_j) \in \mathbb{R}^2$, $j \geq 1$. We have
$$
\|y\|^2 = \sum_{j=1}^{\infty} \lambda_j (\alpha_j^2 + \beta_j^2).
$$
Separation of variables yields
$$
S(s)y = \sum_{j=1}^{\infty}
\begin{pmatrix}
\alpha_j \cos(\lambda_j^{\frac{1}{2}} s) + \beta_j \sin(\lambda_j^{\frac{1}{2}} s) \\
\\
-\alpha_j \lambda_j^{\frac{1}{2}} \sin(\lambda_j^{\frac{1}{2}} s) + \beta_j \lambda_j^{\frac{1}{2}} \cos(\lambda_j^{\frac{1}{2}} s)
\end{pmatrix}
\phi_j, \quad \forall s \geq 0.
$$
Then we have
$$
\|BS(s)y\|^2 = \sum_{j=1}^{\infty} \lambda_j \left[ \alpha_j^2 \sin^2(j\pi s) + \beta_j^2 \cos^2(j\pi s) - \sin(2j\pi s) \alpha_j \beta_j \right].
$$
It follows that
$$
\int_0^1 \|BS(s)y\|^2 ds = \frac{1}{2} \sum_{j=1}^{\infty} \lambda_j (\alpha_j^2 + \beta_j^2),
$$
so (observe that $B^2=B=B^*$) the observation assumption  holds with $T = 1$.\\
Then, for all $r < 2$, the feedback control
$$
v_r(t) = - \frac{\displaystyle \int_\Omega a(x) |  z_t(x,t) |^2 dx}
{\|\big(z(\cdot,t),z_t(\cdot,t)\big)\|^r} \mathbf{1}_{\{(z(t), \dot{z}(t)) \neq (0,0)\}}
$$
strongly stabilizes system~(\ref{Ch4wave1}) with the decay estimate
$$
\|(z(\cdot,t), z_t(\cdot,t))\| = O\big(t^{-1/(2-r)}\big), \quad \text{as } t \to +\infty.
$$

\section{  Concluding remarks }

\begin{itemize}

\item
Weak stabilization can be  extended to reflexive state spaces.
However, in non-reflexive spaces, convergence is obtained in the weak-$*$
sense, and the observation condition is replaced by a sequential one
(see [Benzaza et al., 2023]).

\item The main difficulty when dealing with non-reflexive state spaces is
the non-differentiability of the norm, which is essential for establishing
energy estimates.

\item The control function $z \mapsto v_r(z)$, with $r < 2$, considered above is $k$-homogeneous
with $k = 2 - r > 0$.

\item In the case $r = 2$, the control $v_2$ is uniformly bounded with respect to the initial
states and is $0$-homogeneous. This situation will be considered in Part-II of the courses.

\item The case of homogeneous controls of negative order, i.e. $r>2$, requires a different approach to
well-posedness, such as the theory of maximal monotone operators (see, e.g.,
Br\'ezis, 1973). Moreover, this may lead to finite-time stabilization (FTS) (see Part-III of the course).

\end{itemize}

\newpage

\begin{center}
\huge{Part II}

\end{center}

\begin{center}

{\bf  Exponential Stabilization of Bilinear Systems in Banach Spaces}

\end{center}

\vspace{1cm}

{\bf Purpose of the course}

\begin{itemize}
    \item Provide necessary and sufficient conditions for the uniform exponential stabilization of bilinear systems in a Banach state space.
    \item Formulate stabilization assumptions in terms of integral (observation) estimates involving the control operator and the system operator.
        \item Provide an explicit estimate of the convergence rate.
    \item  Present illustrative examples  of bilinear control systems governed by PDEs.

\end{itemize}

\newpage

\section{Introduction}

\underline{\it Motivation}:\\
Many modeling problems naturally arise in non-reflexive state spaces (see e.g. [Rhandi, 1998], [Rhandi and Schnaubelt, 1999], [Fattorini, 1974], [Fattorini, 2001] and [Vieru, 2005]). For instance,

\begin{itemize}
    \item The appropriate state space for age-dependent population dynamics with spatial diffusion is $X = L^1(\Omega)$, since the $L^1$-norm represents the total population size.

    \item For the heat equation, spaces of continuous functions, such as $\mathcal{C}_0(\Omega)$, often provide a more suitable framework. For instance,  in the context of feedback stabilization of the heat equation
    \[
    z_t = z_{xx} + \mathbf{1}_\omega v(x,t),
    \]
    stabilization in the supremum norm is more interesting than in the $L^2$-norm. 

    \item Constraints may also naturally be expressed in non-reflexive settings.

This becomes even more significant when the following control constraints  are imposed:
    \[
    \|v(t)\|_{L^\infty(\Omega)} \leq M, \;  t\ge 0,
    \]
    which is stricter than the usual constraint in the Hilbert space framework $L^2(\Omega)$.
\end{itemize}

\vspace{0.5cm}

\underline{ \it Considered system}:\\
 Consider the following system:
    $$
   {\blue  (BLS)}\quad \frac{dz(t)}{dt} = Az(t) + u(t) B z(t), \quad z(0) = z_0 \in X.
    $$
    \begin{itemize}

    \item $X$ is a  Banach space.
    \item The scalar-valued function $u(t)$ is the control and $z(t)$ (when it exists) is the corresponding mild solution.
    \item The operator $A$ is the infinitesimal generator of a linear $C_0$-semigroup $S(t)$ on $X$.
\end{itemize}

\vspace{0.5cm}

\underline{ \it Some recalls}:

\vspace{0.25cm}

    $\star$  The duality pairing between a space $X$ and its dual $X^*$ is denoted by $\langle \cdot,\cdot \rangle$, where $X^*$ is the set of all bounded linear functionals on $X$.

        \vspace{0.25cm}

        - The pairing between $y \in X$ and $\phi \in X^*$ is denoted by $\langle y, \phi \rangle$.

        \vspace{0.25cm}

       - If $X$ is a Hilbert space, $X^*$ is identified with $X$ and $\langle \cdot,\cdot \rangle$ with the inner product in $X$.

        \vspace{0.25cm}

    $\star$  The duality map $J$ from $X$ to $X^*$ is generally a multi-valued operator; i.e., for each $y \in X$, $J(y)$ is by definition the (nonempty) set of all $\phi \in X^*$ such that
        $$\langle y, \phi \rangle = \|y\|^2 = \|\phi\|^2.$$

    $\star$  An unbounded linear operator $A$ with domain $\mathcal{D}(A)$ is:
        \begin{itemize}
            \item dissipative if for every $y \in \mathcal{D}(A)$ there exists $y^* \in J(y)$ such that $\Re \langle Ay, y^* \rangle \le 0$.\\
Moreover,, if $A$ is the generator of a strongly continuous contraction semigroup, then we have
\begin{equation}
\Re \langle A x, x^* \rangle \le 0
\end{equation}
 for all $x \in \mathcal{D}(A)$ and for an arbitrary $x^* \in J(x)$.

            \item conservative if $\sup_{y^* \in J(y)} \Re \langle Ay, y^* \rangle = 0$,
            \item strongly conservative if $\Re \langle Ay, y^* \rangle = 0$ for all $y^* \in J(y)$.
        \end{itemize}

    $\star$  A $C_0$-semigroup $S(t)$ is:
        \begin{itemize}
            \item a contraction semigroup if $\|S(t)y\| \le \|y\|$ for all $t \ge 0$ and $y \in X$,
            \item a $C_0$-isometric semigroup if $\|S(t)y\| = \|y\|$ for all $t \ge 0$ and $y \in X$,
            \item a $C_0$-isometric group if it is an isometric semigroup and a group.
     \end{itemize}

\vspace{0.5cm}

\underline{ \it Goal}:

$\bullet$  Provide a feedback law such that the closed-loop system admits a unique (mild) solution satisfying
$$
\|z(t)\| \le M e^{-\sigma t} \|z_0\|, \quad \forall t \ge 0,
$$
for some constants $M, \sigma > 0$.\\

$\bullet$  It is well known that a linear system that is null-exactly controllable is also exponentially stabilizable. Moreover, this condition has been shown to be necessary for the exponential stabilization of linear systems (using linear feedback) when the associated semigroup is a group.

$\bullet$ The goal here is to establish similar results for bilinear systems when the associated semigroup is a contraction/isometric semigroup.

\vspace{0.5cm}

\underline{\it Feedback control candidate}:
\\
 In order to formulate the stabilization assumptions in the context of Banach spaces, we formally differentiate $\|x(t)\|^2, $ obtainig at least for contractive semigroup, $$\frac{d}{dt} \|x(t)\|^2 =2\langle \dot{x}(t),f\rangle
$$
$$
\le 2 Re\;\big( u(t) \langle Bx(t),f \rangle\big),\; \forall f\in J(x(t))
$$
In the previous course, we established weak stability and polynomial stabilization in Hilbert spaces using  the family of control laws:
$$
u_r =- \displaystyle\frac{\langle z(t),B(z(t))\rangle }{\|z(t)\|^r},\; \mbox{if} \; z(t)\ne 0\;\ \mbox{and} \; \; u_r=0
\; \mbox{if} \; z(t)= 0,
$$
with $r<2.$\\
Here, we consider the limit case $ r = 2 $. The objective is to prove
 uniform exponential stabilization under a normalized feedback
control involving a gain parameter, which allows one to regulate the
bound of the control input:
$$
u = -\lambda \frac{\langle B(z(t)), J(z(t)) \rangle}{\|z(t)\|^2},
\quad \text{if } z(t) \neq 0,
$$
and
$$
u = 0, \quad \text{if } z(t) = 0,
$$
(where $ \lambda > 0 $ denotes the control gain), under the following assumptions:

\vspace{0.25cm}

$\bullet$ {\it In order to simplify the notation (avoiding the use of complex conjugates in the expression of the feedback control, as well as real parts in may places,..), we assume in the sequel that $ X $ is a {\blue real Banach} space.}

\section{Exponential Stabilization}

\subsection{ Main assumptions }

 ${\blue ({\cal A}_1)}$ The operator  $A$ generates a linear $C_0-$semigroup $S(t)$ of contractions on $X$.

 \vspace{0.25cm}

  ${\blue ({\cal A}_2)}$ The mapping $B_J: y\mapsto \langle B(y),J(y)\rangle$ is single-valued and satisfies
 $$
 |B_J(y)-B_J(z)| \le \mathcal{K}(\|y\|,\|z\|) \|y-z\|,\;\;  \forall y, z \in X,
$$
 where  $\mathcal{K}: \mathbf{R}^{+2} \rightarrow \mathbf{R}^+$  is a function satisfying the   properties:

 \vspace{0.2cm}

 - the function  $\mathcal{K}: \mathbf{R}^{+2} \rightarrow \mathbf{R}^+$ is increasing (in the sense that $0\le s_1\le r_1$ and $0\le s_2\le r_2$
 imply $\mathcal{K}(s_1,s_2)\le \mathcal{K}(r_1,r_2)$),

 \vspace{0.2cm}

 - $\mathcal{K}(s,s)$ is continuous and $\mathcal{K}(s,s)\le C s,\, \forall s\ge 0$ (for some  $C>0$).
 \vspace{0.25cm}

 \begin{enumerate}

\item If $J$ is single-valued and Lipschitz continuous, then the assumption $(\mathcal{A}_2) $ holds. In particular, if $X$ is a Hilbert
space, then $(\mathcal{A}_2) $ holds with $\mathcal{K}(t,s) = c(t + s)$ for some $ c=2\|B\|. $

\item  Note that assumption $(\mathcal{A}_2)$ guarantees the continuity of the mapping\\
$t\mapsto   \langle  BS(t)y,J(S(t)y)\rangle.$ This gives sense to estimate of assumption $(\mathcal{A}_3)$ below.
\end{enumerate}

  ${\blue ({\cal A}_3)}$ The  observation estimate:
 $$
 \int_{0}^{T}|\langle B(S(t)y),J(S(t)y)\rangle|dt\geq    \delta\|S(T)y\|^{2}, \ \forall y\in X\; (T,\delta>0).
$$

\vspace{0.2cm}

$\star$  In the sequel, under the assumption that $B_J$ is single-valued (for example under assumption $(\mathcal{A}_2)$), we denote by $\langle By,J(y)\rangle$ the  value of $\langle By,y^*\rangle$
for any $ y^*\in J(y).$

\vspace{0.2cm}

$\star$  Also, for any functions $t\mapsto \zeta(t)$, we will write $\phi(\cdot)\in J(\zeta(\cdot))$ if $\phi(t)\in J(\zeta(t)),\, \forall
t\ge0.$

\subsection{Necessary conditions }

Let us prove the following lemma:

\begin{lemme}

If the  assumption $({\cal A}_2)$ holds, then the closed-loop operator  $z \mapsto F(z)=\frac{\langle Bz,J(z)\rangle}{\|z\|^2} Bz$ is globally Lipschitz.
\end{lemme}

{\red Proof.}

Let $y,z\in X-(0)$ such that $\|z\|\le \|y\|.$ We have
$$
F(y)-F(z)=\frac{\langle By,J(y)\rangle}{\|y\|^2} By-\frac{\langle Bz,J(z)\rangle}{\|z\|^2} Bz
$$
$$=\frac{\|z\|^2\langle By,J(y)\rangle By-\|y\|^2\langle Bz,J(z)\rangle Bz}{\|y\|^2\|z\|^2}
$$
$$
=\frac{\|z\|^2\big( \langle By,J(y)\rangle By -\langle Bz,J(z)\rangle Bz\big) +\big(\|z\|^2-\|y\|^2\big) \langle Bz,J(z)\rangle Bz}{\|y\|^2\|z\|^2}
$$
$$=\frac{ \langle By,J(y)\rangle By -\langle Bz,J(z)\rangle Bz}{\|y\|^2} +\frac{\big(\|z\|^2-\|y\|^2\big) \langle Bz,J(z)\rangle Bz}{\|y\|^2\|z\|^2}
$$
Then, using assumption $({\cal A}_2),$ we get
$$
\|F(y)-F(z)\|\le  \frac{ |\langle By,J(y)\rangle - \langle Bz,J(z)\rangle | \|By\|}{\|y\|^2}+
\frac{ | \langle Bz,J(z)\rangle| \|By-Bz\|}{\|y\|^2} +
$$
$$\frac{|\|z\|^2-\|y\|^2| |\langle Bz,J(z)\rangle|\| Bz\|}{\|y\|^2\|z\|^2}
$$
$$
\le \frac{ \|B\| {\cal K}  \big(\|y\|, \|z\| \big)\|y-z\| }{\|y\|}+
\frac{ \|B\|^2 \|z\|^2 \|y-z\|}{\|y\|^2} +
$$
$$\frac{\big ( \|y\|+\|z\|\big)|\|z-y\|| \|B\|^2 \|z\|}{\|y\|^2}
$$
Then, using $\|z\| \le \|y\|$ and, once again, assumption $({\cal A}_2)$, we derive
$$
\|F(y)-F(z)\|\le  C \|B\|  \|y-z\| +\|B\|^2 \|y-z\| +
 2\|B\|^2\|z-y\|
$$
$$
=  \big(C \|B\|   +3\|B\|^2 \big)\|z-y\|
$$
Thus $F$ is $L-$Lipschitz with $L=C \|B\|   +3\|B\|^2.$\\

Next we give necessary conditions for exponential stabilization. Let us consider the feedback law:
\begin{equation} \label{Ch5e(t)}
u(t)=-\lambda \displaystyle\frac{\langle B(z(t)),J(z(t))\rangle }{\|z(t)\|^2},\; \mbox{if} \; z(t)\ne 0\;\ \mbox{and} \; \; u(t)=0
\; \mbox{if} \; z(t)= 0,
\end{equation}
where $\lambda>0$ is the  control gain.

\vspace{0.25cm}

\begin{theoreme}

\begin{itemize}

 \item Let assumption {\blue $({\cal A}_2)$} hold and suppose that the semigroup $S(t)$ is of {\blue isometries}.

 \item   If the control (\ref{Ch5e(t)})  uniformly exponentially stabilizes  (BLS) for some $\lambda>0$, then the observation condition {\blue $({\cal A}_3)$} holds.

\end{itemize}

\end{theoreme}

{\red Proof: }

It is clear that, without loss of generality, we may assume $\lambda=1$ in this proof.

Assume that the system  (BLS) closed with $u(t)$ is exponentially stable:
$$\|z(t)\|\le Me^{-\sigma t} \|z_0\|,\forall t\ge 0,\; z_0\in X.
$$

We know from the above lemma, that assumption $({\cal A}_2)$ implies that the function
$$z\mapsto F(z)=\frac{\langle Bz,J(z)\rangle}{\|z\|^2} Bz$$ is globally Lipschitz (with a Lipschitz constant $L$).

{\bf Step 1.} We have
$$
\|F(z(t))\|
$$

$$
\le  \|F(S(t)z_0)\|+  \|F(z(t))-F(S(t)z_0)\|
$$
$$
\le   \|F(S(t)z_0)\|+ L  \|z(t)-S(t)z_0\|
$$
$$
\le  \|F(S(t)z_0)\|+ L  \int_0^t \frac{|\langle Bz(s),J(z(s))\rangle|}{\|z(s)\|^2} \|Bz(s)\| ds
$$
$$
\le   \|F(S(t)z_0)\|+  L   \int_0^t \|F(z(s))\| ds
$$
 Then, Gronwall's  inequality (see e.g. [Corduneanu, 2008]) yields
 $$
\|F(z(t))\| \le   \|F(S(t)z_0)\|+  L   \int_0^t  \|F(S(s)z_0)\| e^{ L (t-s)} ds
$$
$$
\le   \|F(S(t)z_0)\|+  L  e^{ L t} \int_0^t  \|F(S(s)z_0)\|  ds
$$
Integrating over $[0,T]$ with $T>0,$ this gives
$$
\int_0^T \|F(z(s))\| ds\le  \big(1+T L  e^{ L T)}\big) \int_0^T \|F(S(t))z_0\| ds
$$
 Let $z_0 \in D(A)$. Based on the variation of constants formula (VCF),
one may  show that the function $t \mapsto \|z(t)\|^2$ is
Lipschitz continuous (recall that the norm
square $\|\cdot\|^2$ is not necessarily differentiable), which will  show that $\|z(t)\|^2$ is  differentiable for almost every $t$ in $(0,+\infty).$\\
For this end, we will show that $t \mapsto \|z(t)\|$ is Lipschitz on any interval $[0,T], \ T>0.$ \\
For  all  $h, t\in [0,T]$ we have
$$
z(t+h)-z(t)=S(t+h)z_0-S(t)z_0+\int_0^h S(t+h-s) F(z(s)) ds $$
$$+\int_0^t S(t-s)  (F(z(s+h))-F(z(s))  )ds\cdot
$$
We know that  $\|S(t+h)z_0 -S(t)z_0\|\le h \|Az_0\|.$ Then  using  assumptions $(\mathcal{A}_1)-(\mathcal{A}_2),$ we deduce that
$$
\|z(t+h)-z(t)\| \le h \|Az_0\| +h  C_T  \|z_0\| +  L \int_0^t    \|z(s+h)-z(s)\| ds,\; (C_T>0)
$$
which by Gronwall's inequality gives
$$
\|z(t+h)-z(t)\|\le h \:\big (   \|Az_0\| +  C_T  \|z_0\| \big ) e^{ LT}, \;\; \forall (t,h) \in [0,T]^2.
$$
Then taking  the  graph norm:  $\|z\|_{\mathcal{D}(A)}= \big ( \|Az\|^2+ \|z\|^2\big )^{\frac{1}{2}},\; z\in \mathcal{D}(A),$ we get
 the following estimate:
$$
\|z(t+h)-z(t)\|\le M_0 h \|z_0\|_{\mathcal{D}(A)}, \;\; \forall (t,h) \in [0,t_1]^2,
$$
where $ M_0= \big (   1 +  C_T    \big ) e^{ LT}$.
It follows that  $\Vert z(t)\Vert$ is Lipschitz continuous with respect to $t$ on any interval $[0,T]$. Hence, $\|z(t)\|$ is differentiable almost everywhere and the following equality  holds for every $t>0$ (see  \cite{kato1967nonlinear,botsC6,preisC6}):
$$
C(z):=\frac{d}{dt} \|z(t)\|^2+2 \frac{|\langle Bz(t),J(z(t))\rangle|^2}{\|z(t)\|^2}=2\langle Az(t),J(z(t))\rangle
$$
We have
$$
  \sup_{\phi\in J(z(.))} C(z)=\sup_{\phi\in J(z(.))} 2\langle Az(t),J(z(t))\rangle=0
$$
Since $z\mapsto \langle Bz,J(z)\rangle$ is single-valued, so is the function $t\mapsto \langle Az(t),J(z(t))\rangle$.
Hence,  since $A$ is conservative, it follows that
$$
\frac{d}{dt} \|z(t)\|^2+2 \frac{|\langle Bz(t),J(z(t))\rangle|^2}{\|z(t)\|^2}=0
$$
Then we have
$$
\frac{d}{dt} \|z(t)\|^2=
-2 \frac{|\langle Bz(t),J(z(t))\rangle|^2}{\|z(t)\|^2}
$$
$$
 \Rightarrow 2 \int_0^T\frac{|\langle Bz(t),J(z(t))\rangle|^2}{\|z(t)\|^2} dt = \|z_0\|^2-\|z(T)\|^2
$$
Then we have
$$
2 \int_0^T\frac{|\langle Bz(t),J(z(t))\rangle|^2}{\|z(t)\|^2} dt\ge  \big(1-M^2e^{-2\sigma T}\big) \|z_0\|^2
$$
It follows that (recall that $
\|F(z)\|=\frac{|\langle Bz,J(z)\rangle|}{\|z\|^2} \|Bz\|
$)
$$
2 \int_0^T \|F(z(t))\| dt
\ge  \big(1-M^2e^{-2\sigma T}\big) \|z_0\|
$$
and hence  there exist $T, \alpha>0$ such that
$$
\int_0^T \|F(S(t))z_0\| ds \ge \alpha \|z_0\|, \; \forall z_0\in X
$$
The semigroup being of isometries, we have $\|S(t)z_0\|=\|z_0\|,\; \forall t\ge 0.$ Then
$$
 \|F(S(t))z_0\|=\frac{|\langle BS(t)z_0,J(S(t)z_0)\rangle|}{\|S(t)z_0\|^2} \|BS(t)z_0\|
$$
$$
\le \|B\|\frac{|\langle BS(t)z_0,J(S(t)z_0)\rangle|}{\|S(t)z_0\|}
$$
$$
=\|B\|\frac{|\langle BS(t)z_0,J(S(t)z_0)\rangle|}{\|z_0\|}
$$
We conclude that
$$
\int_0^T |\langle BS(t)z_0,J(S(t)z_0)\rangle| dt \ge \alpha \|z_0\|^2
$$
which extends by density to all $z_0\in X.$

\subsection{Sufficient conditions for exponential stabilization: The case where $B_J$ is single-valued  }

Let us consider the feedback law
$$
u(t)=-\lambda \displaystyle\frac{\langle B(z(t)),J(z(t))\rangle }{\|z(t)\|^2},\; \mbox{if} \; z(t)\ne 0\;\ \mbox{and} \; \; u(t)=0
\; \mbox{if} \; z(t)= 0,
$$
where $\lambda>0$ is the  control gain.

\vspace{0.25cm}

\begin{theoreme}

 Let assumptions
$({\cal A}_1) \;\& \;({\cal A}_2) \;\& \;({\cal A}_3) $ hold.

 Then, for   $\lambda>0$ small enough, the control (\ref{Ch5e(t)})  ensures the  uniform exponential stabilization of system  (BLS).

\end{theoreme}

{\red Proof:}

 According to  the above lemma, we have that under assumption  $(\mathcal{A}_2)$, the closed-loop operator generated by the control (\ref{Ch5e(t)})  is globally Lipschitz.
\\ Consequently, there exists  a unique global mild solution $z(t)$ for the system in closed-loop, which is given for any $z_0\in X$ and for all $t\ge 0$ by the variation of constants formula (VCF):
  $$
z(t)=S(t)z_0-\lambda \int_0^t \frac{\langle Bz(s),J(z(s))\rangle }{\|z(s)\|^2} S(t-s) Bz(s) ds\cdot
$$
 Observing that $z_0=0$ implies $z(t)=0$ for all $t\ge 0, $
we can assume in the sequel that $z(t)$ does not vanish.
\\ Moreover, using an argument of density as in the theorem of necessity,  one can also show that, for $y_0\in X$. Indeed, similarly to the case $\lambda=1,$ we can show that  $\Vert z(t)\Vert$ is Lipschitz continuous with respect to $t$ on any interval $[0,T], T>0$. Hence, $\|z(t)\|$ is differentiable almost everywhere and the following equality  holds  (see  \cite{kato1967nonlinear,botsC6,preisC6}):
\begin{equation}\label{Ch5dt}
\frac{d}{dt} \|z(t)\|^2=2\mathcal{R}e \big (\left<\dot{z}(t),\phi\right > \big ),\, a.e. \; t>0,\; \forall \phi \in J(z(t)),
\end{equation}
 According to assumption $(\mathcal{A}_3)$,  the mapping $t\mapsto \mathcal{R}e \big ( \langle  F(z(t)),\phi(t)\rangle \big )$
is continuous for all $\phi(\cdot)\in J(z(\cdot)).$  Thus we can integrate (\ref{Ch5dt}) to get (since $A$ is dissipative)
$$
 \|z(t)\|^2 -  \|z(s)\|^2\le
- 2  \lambda  \int_s^t    \big(  \langle F(z(\tau)),\phi(\tau)\rangle \big ) d\tau, \;  \forall 0 \le s \le t,
 $$
 for all $\phi(\cdot)\in J(z(\cdot)),$ which we may also write as follows.
\begin{equation}\label{Ch5Ineq*}
 \|z(t)\|^2 -  \|z(s)\|^2\le
- 2  \lambda \int_s^t    \big(  \langle F(z(\tau)),J(z(\tau))\rangle \big ) d\tau, \; 0\le s \le t.
\end{equation}
Moreover, it follows from  assumption $(\mathcal{A}_3), $  together with the density of $\mathcal{D}(A)$ in $X$, that   the inequality (\ref{Ch5Ineq*})  extends to all $z_0\in X$.

\begin{equation}\label{Ch5zszt}
 \|z(t)\|^2 -  \|z(s)\|^2\le
- 2 \lambda \int_s^t    \frac{|\langle Bz(\tau),J(z(\tau))\rangle|^2}{\|z(\tau)\|^2}  d\tau, \;   0 \le s \le t.
\end{equation}
Based on the (VCF), we deduce via Gronwall's inequality
$$
\|z(t)\|\le  \|z_0\| e^{\lambda K^2 t},\; \forall t\ge 0 \; (K:=\|B\|).
$$
Now, from the  variation of constants formula (VCF),  we derive   that for all $t\ge T,$ we have
  $$  \|z(t)\| \le \|S(T)z_0\|+ \lambda K^2 \int_0^t  \|z(s)\| ds  $$
  $$
  \le \|S(T)z_0\|+ \lambda K^2 \|z_0\| \int_0^t  e^{\lambda K^2 s} ds
  $$
  $$
  \le \|S(T)z_0\|+  \|z_0\| \big ( e^{\lambda K^2 t}-1\big) ds,
  $$
  which gives for $t=T$
  $$\|z(T)\| \le   \|S(T)z_0\| +  K_\lambda  \|z_0\|.
  $$
  $ K_\lambda\sim \lambda K^2 T=o(1),$ as $\lambda\to 0^+.$
\\
Thus for all $j\ge0,$ we have
\begin{equation}\label{Ch5zS(T)}
  \|z((j+1)T)\| \le    \|S(T)z(jT)\|^2 + K_\lambda \|z(jT)\|\cdot
  \end{equation}
We have used the following estimations:
$$ \|S(t)z_0\|\le \|z_0\| \; \; \mbox{and}  \; \; \|z(t)\| \le \|z_0\| $$
which imply that
$$
  \mathcal{K}(\|S(t)z_0\|,\|z(t)\|) \le \mathcal{K}(\|z_0\|,\|z_0\|)
\le C \|z_0\|,
$$
where $C$ is the constant defined in assumption $(\mathcal{A}_2)$.\\
It follows  from assumption $(\mathcal{A}_2)$  that for all $t\in [0,T], $
$$
 | \left<BS(t)z_0,J(S(t)z_0)\right>  - \left<Bz(t),J(z(t))\right>  | \le C \|z_0\| \; \|z(t)-S(t)z_0\|,
$$
and hence
$$
 | \left<BS(t)z_0,J(S(t)z_0)\right>|\le    \left<Bz(t),J(z(t))\right>  | + C \|z_0\| \; \|z(t)-S(t)z_0\|
$$
$$
\le  \frac{|\langle Bz(t),J(z(t))\rangle |}{\|z(t)\|} \|z_0\|+
 C  \lambda \|B\| \|z_0\|
   \displaystyle\int_0^T \frac{|\langle Bz(t),J(z(t))\rangle |}{\|z(t)\|}  dt.
   $$
Integrating the last inequality and using again that  $\|z(t)\|$ is  non-increasing, we get via H$\ddot{o}$lder's inequality
$$
\displaystyle\int_0^T | \left<BS(t)z_0,J(S(t)z_0)\right>  |  dt \le C_\lambda \|z_0\|
   \big (\displaystyle\int_0^T \frac{|\langle Bz(t),J(z(t))\rangle |^2}{\|z(t)\|^2}  dt\big)^{1/2}
$$
where $ C_\lambda=\big ( 1 +  C \|B\| \lambda\big ) \sqrt {T}\le C_1$ for $\lambda\in (0,1).$

In the sequel, we suppose that $\lambda\in (0,1).$\\
Replacing $z_0$ by $z(jT), j\ge 0$ in the last inequality and using the observation estimate, we deduce that
$$
C_\lambda^2 \|z(jT)\|^2 \displaystyle\int_{jT}^{(j+1)T}
 \frac{|\langle f(z(t),u_0),J(z(t))\rangle |^2}{\|z(t)\|^2}  dt
\ge \delta^2 \|S(T)z(jT)\|^4,\; \forall j \ge0\cdot
$$
This together with $$
 \|z(t)\|^2 -  \|z(s)\|^2\le
- 2 \lambda \int_s^t    \frac{|\langle Bz(\tau),J(z(\tau))\rangle|^2}{\|z(\tau)\|^2}  d\tau, \;   0 \le s \le t,
$$
 gives
$$
  \|z(jT)\|^4-\|z((j+1)T)\|^2\|z(jT)\|^2\ge  \frac{2  \lambda \delta^2}{C_1^2} \|S(T)z(jT)\|^4\cdot
$$
Hence using the following estimate (see inequality (\ref{Ch5zS(T)})):
 $$ \|z((j+1)T)\| \le    \|S(T)z(jT)\|^2 + K_\lambda \|z(jT)\|, $$
 and the fact that
 $$\|z((j+1)T)\| \le \|z(jT)\|,$$ it comes
$$
\begin{array}{ccc}
  \|z(jT)\|^4-\|z((j+1)T)\|^4&\ge & \frac{2  \lambda \delta^2}{C_1^2} \|S(T)z(jT)\|^4 \\
   & \ge & \frac{  \lambda \delta^2}{4C_1^2}  \bigg (\|z((j+1)T)\|^4 -K_\lambda \|z(jT)\|^4\bigg )
\end{array}
$$
Hence $$ \|z((j+1)T)\|^4\le \zeta \|z(jT)\|^4,$$
where $\zeta:=\frac{1+\frac{  \lambda \delta^2}{4C_1^2} K_\lambda}{1+\frac{  \lambda \delta^2}{4C_1^2} } {\blue \in (0,1)}$ for $\lambda>0$ sufficiently small (recall that $K_\lambda=o(1), $ as $\lambda\to 0^+$).\\
Thus
$$\|z(jT)\|\le \zeta^{\frac{j}{4}}\|z_0\|,\; j\ge 0,$$
  gives (recall that  $\|z(t)\|$  is decreasing)
  $$\|z(t)\|\le \|z(jT)\|, \; j=E(t/T).$$
This gives   the estimate
$$
\|z(t)\|\le M e^{-\sigma t} \|z_0\|,\; t\ge 0
$$
with $ \sigma = \frac{-\ln \zeta}{4T}$ and $M=\zeta^{-1/4}.$\\

\begin{remarque}
   Note that the above results still are valid in the context of complex Banach space using the conjugate of the above normalized feedback control.
\end{remarque}

\subsection{ Sufficient conditions for exponential stabilization: The case where $B_J$ is not single-valued }

In case where  $B_J$ is not single-valued, one may replace assumption $({\cal A}_2)$ by the following one:

\medskip
\noindent
$({\cal A}_2)'$
There exists a linear bounded operator ${\cal B} : X \to X$ such that
${\cal B}_{J}$ satisfies $({\cal A}_2)$ and
$
B \ge {\cal B}, $  in the sense that
 $$
\langle B y, J(y) \rangle \ge \langle {\cal B} y, J(y) \rangle,
\qquad \forall y \in X.
$$
In other words,
$$
\langle B y, y^* \rangle \ge \langle {\cal B} y, y^* \rangle \ge 0,
\qquad \forall y \in X, \ \forall y^* \in J(y).
$$
 Notice that the above positivity condition may be replaced by the condition that the two terms of the first inequality  have the same sign, thereby recovering the above result when $B_{J}$ is single-valued.

\vspace{0.25cm}

Using the same arguments as above, we obtain the following result:

\medskip

\begin{theoreme}

Assume that $({\cal A}_1)$, $({\cal A}_2)'$, and $({\cal A}_3)$ hold,
with $B$ replaced by ${\cal B}$.\\
Consider the control law
$$
u=-\lambda \frac{\langle {\cal B} z(t), J(z(t)) \rangle}{\|z(t)\|^2},
\qquad \text{if } z(t) \neq 0,
$$
and
$$
u=0, \quad \text{if } z(t)=0,
$$
where $\lambda>0$ is the control gain.\\
Then, for $\lambda>0$ small enough, the control $u$ ensures the
 uniform exponential stabilization of system~(BLS).

 \end{theoreme}

{\red Idea of the proof:}

Following the same method as for the lemma, we can see that the new closed operator $$
z\mapsto \frac{\langle {\cal B}z,J(z) \rangle}{\|z\|^2} Bz
$$
is Lipschitz.\\
Moreover, the estimate (\ref{Ch5zszt}) becomes
$$
 \|z(t)\|^2 -  \|z(s)\|^2\le
- 2 \lambda \int_s^t    \frac{\langle {\cal B} z(\tau),J(z(\tau))\rangle\; \langle Bz(\tau),J(z(\tau))\rangle}{\|z(\tau)\|^2}  d\tau, \;   0 \le s \le t.
$$
$$
\le
- 2 \lambda \int_s^t    \frac{|\langle {\cal B} z(\tau),J(z(\tau))\rangle|^2}{\|z(\tau)\|^2}  d\tau, \;   0 \le s \le t.
$$
It follows  from assumption $(\mathcal{A}'_2)$  that for all $t\ge0, $
$$
 | \left<{\cal B}S(t)z_0,J(S(t)z_0)\right>|\le   C \|z_0\| \; \|z(t)-S(t)z_0\| +  | \left<{\cal B} z(t),J(z(t))\right>  |,
$$
where $C$ is the constant defined in assumption $(\mathcal{A}'_2)$.

Then we proceed as in the above theorem.

\begin{remarque}

\begin{enumerate}

\item Note that in assumptions $({\cal A}_2)$ and $({\cal A}_3)$, one can replace $J$ with a subset $J_1 \subset J$, provided that it is the same in both assumptions.

\item In the case of strong observation inequality (exact controllability),  we can prove that the UES hold for any $\lambda>0$.

\item The contraction of the semigroup  is not necessary. For instance, consider
$A = \mu I$, $\mu > 0$, $B = I$, and the above normalized feedback law with $\lambda > \mu.$
However, the stabilization results still hold for some classes of unstable bilinear abstract systems (see [Chen \& Tsao, 2000] and [Ouzahra, 2021]).

\item In the finite-dimensional case (e.g., $X=R^n$), the contraction condition may be assumed to hold with
respect to another inner product
$$
\langle x, y \rangle_{P} := \langle P x, y \rangle,
$$
where the matrix $P = P^{T}>0$ satisfies the LMI: $PA+A^T P\le 0$.

\end{enumerate}
\end{remarque}

\begin{exemple}

Let $\Omega$ denote a bounded open subset of $\mathbf{R}^{n}$ with smooth boundary $\partial\Omega$, and let $Q = \Omega \times (0,+\infty)$.\\
Consider the following wave equation:
\begin{equation}\label{Ch1wave1}
\left\{
\begin{array}{ll}
\displaystyle \frac{\partial^2 z(x,t)}{\partial t^2} = \Delta z(x,t) + v(t) a(x) \frac{\partial z(x,t)}{\partial t}, & \text{in } Q, \\
\\
z = 0, & \text{on } \partial \Omega \times (0,+\infty),
\end{array}
\right.
\end{equation}
where $a \in L^{\infty}(\Omega)$ satisfies $a(x) \ge 0$ a.e. in $\Omega$, and $a(x) \ge c > 0$ on a non-empty open subset $\omega \subset \Omega$.\\
It is well known that the observability estimate holds if $\omega$ satisfies the Geometric Control Condition (GCC), which we assume throughout this example.\\
This system can be written in the abstract form of equation~ (BLS) by setting
$$
H = H_0^1(\Omega) \times L^2(\Omega), \quad
\langle (y_1, z_1), (y_2, z_2) \rangle = \langle y_1, y_2 \rangle_{H_0^1(\Omega)} + \langle z_1, z_2 \rangle_{L^2(\Omega)},
$$
and
$$
A = \begin{pmatrix} 0 & I \\ A_1 & 0 \end{pmatrix}, \quad
B = \begin{pmatrix} 0 & 0 \\ 0 & a(\cdot) \end{pmatrix},
$$
where $A_1 = \Delta$ is the Laplacian operator with domain $D(A) = H^2(\Omega) \cap H_0^1(\Omega)$.\\
Here, $A$ is skew-adjoint and $B$ is linear and bounded.\\
We deduce that the feedback control
$$
u(t) = - \lambda \frac{\displaystyle \int_\Omega  |  z_t(x,t)  |^2 dx}
{\|\big( z(\cdot,t),z_t(\cdot,t)\big)\|^2} \mathbf{1}_{\{(z(t), \dot{z}(t)) \neq (0,0)\}}
$$
($\lambda>0$) exponentially stabilizes system~(\ref{Ch1wave1}).

\begin{remarque}
  In the above example of the undamped wave equation, the Geometric Control Condition (GCC) is assumed to be satisfied to ensure that the exact (strong) observability condition holds. In the absence of the (GCC), only a logarithmic asymptotic estimate of the state is derived with $v_0$,  under a weaker observability condition (see Ammari and Ouzahra, 2021).

\end{remarque}

\end{exemple}

\begin{exemple}
Let   $\Omega =\left( 0,+\infty\right)$ and let us consider the following system
\begin{equation}\label{Ch5eqution transport}
\left\{
\begin{array}{ll}
  y_t(\cdot,t) =   -y_{x}(\cdot,t)+  u(t) h(x) y(\cdot,t),& \mbox{in} \; \Omega\times (0,+\infty) \\
  y(0,t)=0, & \mbox{in} \, (0,+\infty)\\
  y(\cdot,0)=  y_0\in L^1(\Omega), & \mbox{in} \;\Omega
\end{array}
\right.
\end{equation}
Here, $X = L^1(\Omega)$ is the state space, the parameter $u(t)$ represents the control,
and the corresponding solution $z(t) := y(\cdot,t) \in X$ is the state.\\
Moreover, we assume that $h \in L^\infty(\Omega)$, and we define the control operator by
$$
By = h(x)\,y, \quad \forall u \in U := \mathbb{R}, \ \forall y \in X.
$$
The duality map is a multivalued function and for all $y\in X$, we have
$$J(y)=\{\xi\in L^\infty(\Omega) : \; \xi(x)\in  \,\, \mbox{sign}(y(x)) \cdot \|y\|\},$$
 where  the sign function is defined by  sign $(s)=
 \left\{
   \begin{array}{cc}
     1, & s>0 \\
    I, & s=0 \\
     -1, & s<0
   \end{array}
 \right. \;\;$  with $I=[-1,1]$.
\\
 The unbounded operator
 $$A=-\frac{\partial }{\partial x}  \;\mbox{with domain}\; D(A) =\left\{ y\in W^{1,1}\left( \Omega \right); \; y\left( 0\right)
=0\right\} $$
  generates an isometric semigroup $S(t)$ on  $X, $ which is defined  for all $y\in L^1(\Omega)$ by
$$
S(t) y(\xi)=\left\{
  \begin{array}{ll}
    y(\xi-t), & \mbox{if} \; \xi-t\ge 0 \\
\\
    0, & \mbox{else.}
  \end{array}
\right.
$$
Hence,  a necessary condition for uniform exponential stabilization  with the control (\ref{Ch5e(t)}) is
given by
$$
\int_0^T \int_\Omega |h(x)| |(S(t)y)(x)|   dx  dt  \ge \delta \|y\|, \; (T, \delta>0).
$$

This last inequality  holds if in particular  $h\ge c>0$ but is not necessary.\\

Let us assume in the sequel that $h={\bf 1}_\omega$ with $\omega=(\alpha,+\infty),\; \alpha>0.$\\
Then, for $ y\in X$ and $y^*\in J(y)$, we have
$$
\langle By,y^*\rangle = \|y\|  \int_\alpha^{+\infty} |y(x)| dx
$$
so $({\cal A}_2)$ holds (in particular $B_J$ is single-valued) for $\mathcal{K}(t,s)=2  (t+s)$ and $C=4.$\\
Moreover, for every $t\ge 0, $  we have
 $$\begin{array}{ccc}
    \langle  BS(t)y,J(S(t)y)\rangle
     &= & \|S(t)y\| \int_\alpha^{+\infty}  |\big(S(t)y\big)(x)| dx.
     \end{array}
$$
It follows that for any $T>\alpha$
  $$\begin{array}{ccc}
     \int_0^T\langle  BS(t)y,J(S(t)y)\rangle dt
& \ge & \int_\alpha^T \|S(t)y\| \int_t^{+\infty}  |y(x-t)| dx\\
\\
& = &c \|y\|^2,\; c=T-\alpha.
  \end{array}
$$
Hence the system (\ref{Ch5eqution transport}) is uniformly exponentially
stabilizable with  the  control
$$u(t)=-\lambda \frac{\int_\alpha^\infty |y(x,t)|dt }{\int_0^\infty |y(x,t)|dt} {\bf 1}_{\{ y(\cdot,t)\ne0\}}$$ for $\lambda>0$ {\blue small enough}.

\end{exemple}

\begin{exemple}
Consider the following system:
\begin{equation}  \label{heat1} \left\{%
	\begin{array}{ll} x_{t}(\cdot,t)= x_{\zeta\zeta}(\cdot,t)+u(t) a(.)x(.,t)   &\\
	x'(0,t)=x'(1,t)=0 &\\
	x(\cdot,0)=  x_0\in \mathcal{ C}_0([0,1]) &
	\end{array}
	\right.
\end{equation}
where

$\star$ the state space is $X=\mathcal{C}_0([0,1])$, equipped with the supremum norm,

$\star$ the scalar valued function $t\mapsto u(t)\in \mathbf{R}$ represents  the control and $x(t)=x(\cdot,t)$ is the corresponding state,

$\star$ the function $a$ is such that  $a\in \mathcal{ C}_0(0,1)$    with $a(\zeta)\geq k>0,\; \forall \zeta\in [0,1]. $ \\

$\star$ Here, the  system operator $A$ is given by $Ax=x_{\zeta\zeta},$ $\forall x\in \mathcal{ D}(A)=\{x\in \mathcal{C}^2([0,1]): \; x'(0)=x'(1)=0\}$   and generates a contractions $C_0-$semigroup
$S(t)$ in $X:=\mathcal{C}_0([0,1])$.\\

Here, we can take:

$\star$ The sub-mapping $J_1: x \mapsto J_1(x)$ of $J$
$$ J_1(x) := \{\varphi= x(s_0) \delta_{s_0}: \;  s_0\in [0,1] \; \mbox{is s.t } \;   |x(s_0)|=\|x\|\},$$   where $\delta_{s_0}$ is any point measure  supported by a point $s_0$ where $|x|$ reaches its maximum.\\
$\star$ For all $x\in X,$ we have
\begin{eqnarray*}
\langle B x,J_1(x)\rangle&=& \langle ax,J_1(x)\rangle\\
&=&\{\langle a x,x(s_0)\delta_{s_0}\rangle/\;s_0\in [0,1] \; \mbox{is s.t } \;   |x(s_0)|=\|x\|\}
\\&=&\{a(s_0)\|x\|^2/\; s_0\in [0,1] \; \mbox{is s.t } \;   |x(s_0)|=\|x\|\}.
\end{eqnarray*}
Here $\langle B x,J_1(x)\rangle$ is a multi-valued function, but we observe that
$$
\langle B x,J_1(x)\rangle\ge  k \|x\|^2.
$$
Thus we can take  the single-valued operator: ${\cal B}=k\:I\!d.$\\
$\star$ For $T>0$ and $x_0\in D(A)$, we have
$$\displaystyle\int_{0}^{T}\left\langle {\cal B} S(t)x_0,J_1(S(t)x_0)\right\rangle dt\geq k \int_{0}^{T}\|S(t)x_0\|^2dt$$
which gives (since $\|S(t)x_0\|$ decreases)
$$\displaystyle\int_{0}^{T}\left\langle {\cal B} S(t)x_0,J_1(S(t)x_0)\right\rangle dt\geq k T\|S(T)x_0\|^2.$$
This gives $({\cal A}_3).$\\
$\star$  Finally, for all $x, z\in X_1$ we have
$$\begin{array}{llll}
	|\langle {\cal B} x,J_1(x) \rangle-\langle Bz,J_1(z)\rangle | &\leq& \|a\|_\infty \big|\|x\|^2-\|z\|^2\big| \\
\\
&\leq& \|a\|_\infty \big( \|x\|+\|z\|\big ) \|x-z\|.&
\end{array}
$$
Thus $({\cal A}_2)$ holds for ${\cal K}(t,s)=\|a\|_\infty \big (t+s\big).$
\\
$\star$ We conclude that the feedback control
$$
u(t) = -\lambda \, a(s_0)\, \mathbf{1}_{\{x(\cdot,t)\neq 0\}},
$$
where $|x(s_0,t)|=\|x(\cdot,t)\|_\infty$, yields an exponentially stable
closed-loop system for a sufficiently small gain $\lambda>0$, with respect
to the supremum norm of $\mathcal{C}_0([0,1])$.

\end{exemple}

\newpage

\begin{center}
\huge{Part III}

\end{center}

\begin{center}

{\bf  Stabilization in Finite-Time of Bilinear Systems
 }

\end{center}

\vspace{1cm}

{\bf Outline}

\begin{itemize}
    \item Introduction

    \item Finite-time stabilization of bilinear systems (multiplicative control):

    $\star$ A necessary condition for finite-time stabilization (FTS)

        $\star$  Sufficient conditions for FTS

    \item Fixed-time (FxTS) and prescribed (PrTS) stability

    \item Stabilization in finite-time of linear systems (additive control)

    \item Applications
\end{itemize}

\newpage

\section{Introduction}

 The concept of stability (introduced by A.\ M.\ Lyapunov in 1892) is one of the central notions in control theory and describes the system's response to small perturbations of the initial conditions. The notion of asymptotic stability in control theory implies convergence of the system trajectories to an  equilibrium state over an {\color{blue} infinite horizon}. In many applications, it is desirable that the system trajectories converge to the equilibrium state in {\color{blue} finite time}, rather than only as $t \to +\infty$.

 The notion of finite-time stability  was first introduced in the finite-dimensional setting by Roxin, (1966).  Earlier contributions appeared in the Russian literature between 1951 and 1954 (see, e.g.,  Erugin (1951),  Kamenkov (1953), and Lebedev (1954)). This concept has since been extensively developed by many researchers in both finite- and infinite-dimensional settings:

\vspace{0.25cm}

$\star$ Finite-stability (without controls): [Bhat, S. P., $\&$ Bernstein, D. S. (2000), (2005)], [Dorato, P. (2006)], [Moulay, E., \& Perruquetti,  (2006)]...

\vspace{0.2cm}

$\star$ Finite-time stabilisation of abstract Eqs: [Polyakov et al. (2018)],  [Amaliki, Y., \& Ouzahra, M. (2025)], [Jammazi et al. (2023)]...

\vspace{0.2cm}

$\star$ FTS of PDEs: [Coron, J. M., $\&$ Nguyen, H. M. (2017)], [Barbu, V. (2018)], [Espitia et al. (2019)], [Zhang, C. 2019], [Ouzahra, M. (2021)]...

\section{Notions of FTS/FxTS/PrTS}

\subsection{Two elementary examples}

{\bf \it    Example 1.}\\
Let
$$
({\blue E_0}) \hspace{0.75cm} \dot{x}(t) = u(t)x(t)
$$
 Consider the control: $ u = -|x|^{-2\mu} {\bf 1}_{(x\ne0)},\; 0<\mu <1/2$.\\
 For $ V(x) = x^2(t) $, we have
$$
{\blue \dot{V} = -2V^{1 - \mu }}.
$$
Integrating over the maximal interval on which $ x(t) \neq 0 $ yields:
$$
 V^{\mu}(t) = \left\{
                    \begin{array}{ll}
                      V^{\mu}(0) - 2 \mu t, & \text{if} \; t < \frac{V^\mu (0)}{2\mu} \\
                      0, & \text{if} \; t \geq \frac{V^\mu (0)}{2\mu}
                    \end{array}
                  \right.
$$
This yields the {\blue FTS}: $x(t)=0,\; \forall t\ge  T := \frac{\|x_0\|^{2\mu}}{2\mu}$.

\vspace{1cm}

{\bf \it Example 2.}\\
Let
$$
({\blue E_1}) \hspace{0.75cm} \theta_t = \theta_{xx} + u(t) \theta
$$
with Dirichlet boundary conditions (DBC).\\
 The equation $ ({\blue E_1}) $ is FTS under the control
    $$
    u(t) = -\|\theta(t)\|^{-2\mu} {\bf 1}_{(\theta(t)\ne0)}, \quad 0 < \mu < 1/2,
    $$
 This gives the FTS: $\theta(t)=0,\; \forall t\ge  T := \frac{\|\theta(0)\|^{2\mu}}{2\mu}$.
\\
 The approach relies on the following Lyapunov functions:

\vspace{0.2cm}

$\rightsquigarrow$ ${\blue V(x)=\|x\|_d}$ which, in this situation, turns out to be $V(x)=\|x\|^m, \; 0<m<0$ (where $\|\cdot\|_d$ is an homogeneous norm)

 (see {\bf  Polyakov, A., Coron, J. M., $\&$ Rosier, L. (2018)}).

\vspace{0.2cm}

$\rightsquigarrow$ ${\blue V(x)=\|x\|^2}$ (see {\bf Ouzahra, 2021}).

\subsection{Considered system and definitions}

Let us consider the system:
$$
    \textcolor{blue}{(\text{BLS})} \hspace{1cm} \dot{x}(t) = Ax(t) + \textcolor{red}{u(t) Bx(t)}, \quad x(0) = x_0,
    $$
where

\begin{itemize}

  \item    $H$ is a Hilbert state space with inner product $\langle \cdot , \cdot \rangle$ and corresponding norm $\|\cdot\|$.

\item The system operator $A$ generates a $C_0$-semigroup $S(t).$

\item The operator $ B $ is a \textcolor{blue}{bounded linear} operator on $ H $, such that $ B^* = B \geq 0 $.

   \item The function $ u(t) \in \mathbb{R} $ is the  control.

\end{itemize}

\vspace{0.5cm}

Let us recall the notion of stabilization in finite time.\\

\vspace{0.2cm}

{\bf Definition.}

 {\it $\bullet$ A control system ({\blue BLS})  is said to be {\blue finite-time stabilizable} (FTS) at the origin ($0$ is an equilibrium of the uncontrolled system; i.e. $u=0$ ) if:

\vspace{0.2cm}

$\rightsquigarrow$  there exists a feedback control such that the system in closed-loop admits a unique solution such that the corresponding closed-loop system satisfies the following properties:

\vspace{0.2cm}

$\rightsquigarrow$ $(0)$ is a Lyapunov stable equilibrium (i.e.,  every trajectory starting inside the $\alpha-$ball remains inside
the $\varepsilon-$ball for all future time),

\vspace{0.2cm}

$\rightsquigarrow$ $\exists T = T(x_0) > 0$ such that $x(t) = 0, \quad \forall t \ge T.$

\vspace{0.2cm}

 In this case, the time $T_*(x_0) := \inf \{ t > 0 : x(t) = 0 \}$ is called the {\blue settling time}.

\vspace{0.2cm}

$\bullet$ If, in addition, $\sup_{x_0} T_*(x_0) < +\infty$, we say that the system is
 {\blue fixed-time stable (FxTS)}.

 $\bullet$ In the case where the settling time can be chosen a priori and arbitrarily, i.e. if for any given $T>0$ there exists a feedback law for which the above properties of FTS are satisfied, then the system is said to have the {\blue stability in prescribed time   (PrTS)}.

}

\begin{remarque}

$\star$ If the closed-loop system possesses the property of backward uniqueness in time, then finite-time stabilization (FTS) cannot occur.

\vspace{0.25cm}

$\star$ In particular, if the system is reversible in time, then there exists no control $u$ for which the mapping $f:x\mapsto u(x)Bx$ is locally Lipschitz and the system exhibits finite-time stabilization (FTS).

\vspace{0.25cm}

$\star$ See [Bardos and Tartar, 1973] for the backward uniqueness of certain classes of parabolic equations.

\end{remarque}

\subsection{ Necessary conditions for FTS/FxTS/PrTS}

\begin{theoreme}
 $\bullet$ {\blue If} the system $ (BLS) $ is FTS, {\blue then}
    $$
    \forall \xi \in H, \quad BS(t) \xi = 0, \quad \forall t \geq 0 \Rightarrow \; \exists t_1 = t_1(\xi) > 0 \;/\; S(t_1) \xi = 0.
    $$
$\bullet$ {\blue If} the system $ (BLS) $ is FxTS, {\blue then} there exists $ t_1  > 0$, such that
    $$
    \forall \xi \in H, \quad BS(t) \xi = 0, \quad \forall t \in [0,t_1] \Rightarrow \; \; S(t_1) \xi = 0.
    $$

$\bullet$ {\blue If} the system $ (BLS) $ is PrTS, {\blue then} for every $\tau>0,$ we have
    $$
       \forall \xi \in H, \quad BS(t) \xi = 0, \quad \forall t \in [0,\tau] \Rightarrow   \xi = 0.
    $$

\end{theoreme}

\textbf{Proof.}

 We will treat the case of {\blue FTS}, as the other cases can be done similarly.
\\
 Assume that the system {\blue (BLS)} is globally finite-time stable, and let $x_0\in H$ be such that $BS(t)x_0=0$ for all $t\ge 0$.
\\
 Then the unique solution of system (BLS) with initial condition $x(0)=x_0$ is given by
$
x(t)=S(t)x_0,\; t\ge 0,
$
independently of the choice of the feedback control $u(x)$.\\
  Moreover, because the system is globally finite-time stable, there exists a finite
time $t_1=t_1(x_0)\ge 0$ such that
$
S(t)x_0=x(t)=0, \; \forall\, t\ge t_1.$
  This concludes the proof.
\vspace{0.25cm}

\begin{remarque}
  If $S(t)$ is one to one (which is the case for instance if $S(t)$ is self-adjoint), then

  $\star$ the above necessary condition of  FTS is equivalent to
  $$
 {\blue (OBS)}
    \hspace{1cm}  \forall \xi \in H, \quad BS(t) \xi = 0, \quad \forall t \ge 0 \Rightarrow   \xi = 0.
    $$
    $\star$ the  necessary condition of   FxTS is equivalent to
    $$
  \forall \xi \in H, \quad BS(t) \xi = 0, \quad \forall t \in [0,t_1] \Rightarrow   \xi = 0.
    $$

\end{remarque}

Building on the general necessary condition presented earlier, the following corollary establishes a particular case of the above necessary condition.

\begin{corollaire}\label{prop1}
Assume that the system (BLS) is globally finite-time stable.

If {\blue $\ker(B)$ is invariant under $S(t),$ } then the  semigroup $S(t)$ induces on the subspace $\ker(B)$ a nilpotent semigroup
$S(t)_{\vert \ker(B)}$.
\end{corollaire}

{\bf Proof.}

 If the system (BLS) is finite-time stable, then there exists a feedback control $u(x)$ such that, for every $x_0\in \mathcal{H}$, the corresponding solution satisfies, for some $T_{x_0}\ge 0$
$$
x(t)=0, \qquad \forall t\ge T_{x_0}.
$$
  Moreover, since $\ker(B)$ is invariant under $S(t)$, we see that for any
$x_0\in \ker(B)$, the unique solution of (BLS), independently of the
stabilizing control, is
$$
x(t)=S(t)_{\vert \ker(B)}x_0, \qquad t\ge 0.
$$
Hence,
$$
S(t)_{\vert \ker(B)}x_0=0, \qquad \forall t\ge T_{x_0}.
$$
 Let $n_0=\lfloor T_{x_0}\rfloor+1$. Then
$$
\big(S(1)_{\vert \ker(B)}\big)^{n_0}x_0=0.
$$
  Let us consider the set of closed subspaces
$$
E_k=\{x\in \ker(B):\; \big(S(1)_{\vert \ker(B)}\big)^{k}x=0\},\; k\in \mathbb{N}.
$$
By assumption, we have $\ker(B)=\cup_{k\ge 0} E_k.$

We conclude that

$\star$  In the infinite-dimensional setting, Baire's theorem implies that
$S(t)_{\vert \ker(B)}$ is nilpotent.

\vspace{0.25cm}

$\star$ In the finite-dimensional state-space case, the same conclusion follows from arguments specific to finite-dimensional spaces. However, in this setting, the property holds only when $\ker(B)=\{0\}$, since the exponential of a matrix is never nilpotent on a nontrivial invariant subspace.

\begin{remarque}

If $\ker(B)$ is invariant under $S(t)$, then we have the following remarks:

\vspace{0.25cm}

$\star$ If $S(t)$ is injective  and if  $\ker(B)\neq \{0\}$, then for any initial
condition $x_0\in \ker(B)\setminus\{0\}$, the system
(BLS) cannot achieve finite-time stability regardless
the choice of the control law $u$.

\vspace{0.25cm}

$\star$ This limitation prevents us from considering global finite-time stability in certain classes of bilinear systems, including {\blue commutative systems}.

\end{remarque}

\subsection{  Description of the approach and main assumptions}

\begin{itemize}

\item We first provide a stabilizing feedback control candidate $ u$ for (BLS), which leads to a Lyapunov function $ {\blue V \geq 0} $ for the system in closed-loop, in the sense that:
$ {\blue \dot{V} \leq -m V^\alpha} $, along the trajectory $ x(\cdot) $, for some $ m > 0 $ and $ 0 < \alpha < 1 $.

\vspace{0.15cm}

 \item Differentiating  formally,  leads to the following estimate (valid at least for a dissipative system operator):

$$
\frac{d}{dt} \|x(t)\|^2 \le 2u(x(t)) \langle Bx(t), x(t)\rangle.
$$

\vspace{0.15cm}

\item Inspired by the two examples above, one can consider the following feedback control (at least when $B^*=B\ge 0$):
$$ {\blue u(x) = - \langle Bx, x \rangle^{-\mu}},\; 0<\mu <1/2.$$

\end{itemize}

\begin{itemize}

\item The above control leads to the following inequality: $$ {\blue
\frac{d}{dt} \|x(t)\|^2 \le -2  |\langle Bx(t), x(t)\rangle|^{1-\mu}}.
$$

    \item Two types of Lyapunov function candidates arise:
    $$
   {\blue V(x) = \|x\|^2} \quad \text{and} \quad {\blue V(x) = \langle Bx, x \rangle}.
    $$

\vspace{0.25cm}

\item In the sequel, we assume that {\blue $B^*=B\ge 0$} on $H$.

\end{itemize}

\begin{itemize}

\item We will first  consider  the FTS of the system \textcolor{blue}{(BLS)} on {\blue invariant closed} subspace $F$ of $H$, on which $B=B^*$ is coercive:
    $$
    \textcolor{blue}{B \ge \beta I} \quad \text{on} \quad \textcolor{blue}{F}, \quad (\beta > 0).
    $$

    \item  Using  the fact that $B^*=B\ge 0, $ it is straightforward to verify that if $ B $ is \textcolor{blue}{coercive on a closed, $ B $-invariant subspace} $ F \subset \mathcal{H} $, then {\blue $ F \subset \ker(B)^\perp $}.

     In other words, the subspace $ \ker(B)^\perp $ is the \textcolor{blue}{largest subspace} on which $ B $ can be coercive.

    \vspace{0.25cm}

    \item We will therefore study FTS on $ \ker(B)^\perp $, and then provide situations in which, one deduce the FTS on the entire state space in both cases: $ \ker(B)^\perp $ is or not invariant under $S(t)$.

\end{itemize}

\vspace{0.15cm}

{\bf Main Assumptions. }

\vspace{0.5cm}

$\bullet$ $  ({\cal A}).$ The system operator $A$ generates a  $C_0-$semigroup of {\blue contractions} $S(t)$.

\vspace{0.25cm}

$\bullet$ $  ({\cal B}).$  {\blue $B^* = B\ge 0$ with $B\ge \beta I$} on {\blue $\ker(B)^\perp$}, \; $(\beta>0)$.

\vspace{0.5cm}

$\star$ In the context where $ B^* = B \geq 0 $, the above coercivity condition is equivalent to the fact that $ B $ has a closed range.

\vspace{0.25cm}

$\star$ Note that the contraction assumption in $({\cal A})$ can be replaced by the quasi-contraction property, which, together with the second part of assumption $({\cal B})$, is automatically satisfied in the {\blue finite-dimensional case}.

\vspace{0.25cm}

$\star$ The assumption $({\cal B})$ is also satisfied in the case of local action: $B=\chi_\omega$ on $H=L^2(\Omega),...$

\subsection{Stabilization in finite time: The case where $ \ker(B)^\perp$ is invariant under $S(t)$ }

   Here, we take the Lyapunov function candidate: $V(x)=\|x\|^2$ and assume that $\ker(B)^\perp$
is invariant under $S(t)$, which will guarantees  the well-posedness of the subsystem
of (BLS) induced on $\ker(B)^\perp$.

\begin{theoreme}
 Assume that the assumptions $ (\mathcal{A}) $ and $ (\mathcal{B}) $ hold, and that
    {\blue $ \ker(B)^\perp $ is invariant under the semigroup $ S(t) $}.\\
     Then  the system \textcolor{blue}{(BLS)} is finite-time stable on \textcolor{blue}{$ \ker(B)^\perp $}.   Moreover, the stabilizing control is given by
$$ {\blue u(x) =   - \langle Bx, x \rangle^{-\mu}}, $$
 with $u(x)=0,$ if $x=0$, and $ 0 < \mu < 1/2 $.\\
 Furthermore, if $ \ker(B)^\perp $ is absorbent, in a finite time $\tau>0$, for the semigroup $S(t)$, then we get the global FTS using the control
$$
u=\left\{
    \begin{array}{ll}
      0, \; 0\le t\le \tau \\
\\
      - \langle Bx, x \rangle^{-\mu} {\bf 1}_{(x\ne0) }, \; t>\tau
    \end{array}
  \right.
$$

\end{theoreme}

{\bf Proof.}

{\bf I- Well-posedness.} First, observe that the space $F=(\ker(B))^{\perp}$ is a Hilbert space as it is a closed subspace of $H$.

The system in closed loop induces a subsystem on $(\ker(B))^{\perp}$ can be written as

\begin{equation}\label{max}
\left\{
\begin{array}{l}
\dfrac{d}{dt} x(t)+\mathcal{A} x(t)=0, \quad t>0,\\
x(0)=x_{0} \in (\ker(B))^{\perp},
\end{array}
\right.
\end{equation}
where the operator
$$
\mathcal{A}: D(A)\cap(\ker(B))^{\perp}\to(\ker(B))^{\perp}
$$
is defined by
$$
\mathcal{A}z=-Az+\Theta(z),
$$
and $\Theta$ is defined for all $z\in(\ker(B))^{\perp}$ by
$$
\Theta(z)=
\begin{cases}
\langle Bz,z\rangle^{-\mu}Bz, & z\neq 0,\\
0, & z=0.
\end{cases}
$$
To establish the well-posedness of system {\color{blue}{(\ref{max})}},
we use an approach based on maximal monotone operators.\\
 Since $S(t)$ is a contraction semigroup, the operator $A$ with
domain $D(A)$ is maximal monotone.
It therefore suffices to prove that $\Theta$ is maximal monotone.\\
  Consider the functional
$$
\varphi(z)=\frac{\langle Bz,z\rangle^{1-\mu}}{2(1-\mu)},
\qquad \forall z\in(\ker(B))^{\perp}.
$$
 Clearly, $\varphi$ is convex and continuous on $(\ker(B))^{\perp}$.\\
Writing $\varphi=h\circ\psi$ with
$$
h(\lambda)=\frac{\lambda^{1-\mu}}{2(1-\mu)},
\quad \lambda\in\mathbb{R},
\qquad
\psi(z)= \langle Bz,z\rangle,
$$
we see that $\varphi$ is Fr\'echet differentiable for every
$z\in(\ker(B))^{\perp}\setminus\{0\}$, and
$$
\varphi'(z)=h'(\psi(z))\psi'(z)
=\langle Bz,z\rangle^{-\mu}Bz.
$$
  For $z=0$, using the self-adjointness of $B$, we verify that
$\lim_{h\to 0}\frac{\varphi(h)}{\|h\|}=0.$
 We conclude that $\varphi$ is differentiable on all of $(\ker(B))^{\perp}$
and $\Theta=\varphi'$. Hence  (see Brezis, 1973), the operator $\Theta$ is maximal monotone.
\\
 Therefore, (see Brezis, 1973), the operator $\mathcal{A}$ generates a nonlinear semigroup and the closed-loop system {\color{blue}{(\ref{max})}} admits a unique weak solution $x(t)$.
\\
 That is,  there exists a sequence
$x_0^n \in D(A) \cap (\ker(B))^{\perp}$ such that $x_0^n \to x_0$, and the system
(\ref{max}) with initial state $x_0^n$ admits a strong solution
$x^n \in C([0,+\infty),(\ker(B))^{\perp})$ satisfying
$x^n \to x$ uniformly on $[0,T]$ for every $T>0$.

\vspace{0.25cm}

{\bf II-FTS}

\vspace{0.25cm}

$\star$  For almost every $t>0$, we have $x^n(t) \in D(A) \cap (\ker(B))^{\perp}$ and
$$
\frac{d}{dt}x^n(t)
= A x^n(t)
-\langle Bx^n(t),x^n(t)\rangle^{-\mu}
B x^n(t)\mathbf{1}_{\{x^n(t)\neq0\}}.
$$
$\star$ To establish the finite-time stability of (\ref{max}), consider the Lyapunov function
$$
V(x)=\|x\|^2 .
$$
Its derivative along the trajectories of (\ref{max}) yields
$$
\frac{1}{2}\frac{d}{dt}\|x^n(t)\|^2
=\langle Ax^n(t),x^n(t)\rangle
-\langle Bx^n(t),x^n(t)\rangle^{1-\mu}.
$$
Using the dissipativity of $A$ together with assumption $ (\mathcal{B}), $ we obtain
$$
\frac{1}{2}\frac{d}{dt}\|x^n(t)\|^2
\le -\beta^{1-\mu} \|x^n(t)\|^{2(1-\mu)}
\quad \text{for a.e. } t\ge0.
$$
Integrating this differential equation over the maximal interval
$[0,\tau)$, $\tau\le +\infty$, on which $x^n$ does not vanish,
we see that $\tau<+\infty$ and that (we can also apply the comparison principle)
 $$
\begin{cases}
\|x^n(t)\|^{2\mu}
\le \|x_0^n\|^{2\mu}-2\mu \beta^{1-\mu}  t,
& t\le \dfrac{\|x_0^n\|^{2\mu}}{2\mu \beta^{1-\mu} },\\
\\
\|x^n(t)\|=0,
& t\ge \dfrac{\|x_0^n\|^{2\mu}}{2\mu \beta^{1-\mu} }.
\end{cases}
$$
Letting $n\to+\infty$, we deduce that
$$
\begin{cases}
\|x(t)\|^{2\mu}
\le \|x_0\|^{2\mu}-2\mu \beta^{1-\mu}  t,
& t\le \dfrac{\|x_0\|^{2\mu}}{2\mu \beta^{1-\mu} },
\\
\|x(t)\|=0,
& t\ge T^*=\dfrac{\|x_0\|^{2\mu}}{2\mu \beta^{1-\mu} },
\end{cases}
$$
giving rise to finite time extinction of the solution.\\
Hence we have the FTS with the following estimation of the settling time:
$$
T_*(x_0)\le \dfrac{\|x_0\|^{2\mu}}{2\mu \beta^{1-\mu} }.
$$

Finally, the GFTS follows by applying the first point to the new initial state $X_0=S(\tau)x_0\in \ker(B)^\perp.$

\subsection{ FxTS $\&$ PrTS: $ \ker(B)^\perp$ is invariant under $S(t)$  }

\begin{theoreme}

     Assume that the assumptions $ (\mathcal{A}) $ and $ (\mathcal{B}) $ hold, and that
    {\blue $ \ker(B)^\perp $ is invariant under the semigroup $ S(t) $}.

     Then

\vspace{0.2cm}

(i) the system \textcolor{blue}{(BLS)} is fixed-time stable on \textcolor{blue}{$ \ker(B)^\perp $}, and the stabilizing control is given by: $$ {\blue u(x) =   - \langle Bx, x \rangle^{-\mu}- \langle Bx, x \rangle^{+\mu}}, $$
 with $u(x)=0,$ if $x=0$, and $ 0 < \mu < 1/2 $.\\
 Moreover, we have
$$
T_*(x_0)\le \dfrac{\pi}{4 \mu \beta^{1-\mu}}.
$$
\end{theoreme}

(ii) If in addition $ \ker(B)^\perp $ is absorbent for $S(t)$ at a $\tau>0,$ then we have the {\blue FxTS on $H$} with the control:
$$
{\blue \bar{u}={\bf 1}_{[0,\tau]} u}
$$
 for $0<\mu<1/2$.
\\  More precisely, we have

$$
x(t)=0, \; \forall t\ge  \tau + \dfrac{\pi}{4 \mu \beta^{1-\mu}}.
$$
{\bf Proof.}\\
 The first part of the proof follows from the next lemma, while the second part of the theorem comes from the same arguments as in the case of FTS.\\

\begin{lemme} (see [Parsegov et al., 2012])

Let $x$ be a solution of the  system
$$
\dot{x}(t) = g(t,x(t)),
$$
where $g : \mathbb{R}_+ \times H \to H$ is a (possibly nonlinear)
function satisfying $g(t,0)=0$ for all $t \ge 0$.
Let $V : H \to \mathbb{R}$ be a continuous, positive definite, and radially unbounded function.

If
$$
\frac{d}{dt} V(x(t))
\le - a V^{1-\nu}(x(t)) - b V^{1+\nu}(x(t)),
$$
for some $a,b>0$ and $0<\nu<\frac{1}{2}$, then
$$
x(t)=0, \quad \text{for all } t \ge T := \frac{\pi}{2\nu\sqrt{ab}}.
$$
\end{lemme}

\begin{remarque}
   If the main inequality of the above lemma holds with $b=0$, then we retrieve the
finite-time stability, with a settling time that depends on the initial condition.

\end{remarque}

\vspace{0.5cm}

\begin{theoreme}

 Assume that the assumptions $ (\mathcal{A}) $ and $ (\mathcal{B}) $ hold, and that
    {\blue $ \ker(B)^\perp $ is invariant under the semigroup $ S(t) $}.
\\
Then  the system \textcolor{blue}{(BLS)} is stablizable in prescribed time on \textcolor{blue}{$ \ker(B)^\perp $}.     Moreover, the stabilizing control is given by $$ {\blue u(x) =  - \rho \langle Bx, x \rangle^{-\mu}- \rho \langle Bx, x \rangle^{+\mu}}, $$
 with $u(x)=0,$ if $x=0$, and   where  $ 0 < \mu < 1/2 $ and $\rho$ is a gain control that can be chosen to reach an a priori given settling time.

\end{theoreme}

{\bf Remark.} In the case of $\omega_0-$quasi-contractive semigroup $S(t)$, the quantity $-\frac{\omega_0}{\beta} $ should be added in the expression of the above feedback laws:\\

$\star$ FTS: the stabilizing control is $$ {\blue u(x) = -\frac{\omega_0}{\beta}  - \langle Bx, x \rangle^{-\mu}}.$$

$\star$ FxTS: the stabilizing control is
$$ {\blue u(x) = -\frac{\omega_0}{\beta}  - \langle Bx, x \rangle^{-\mu}- \langle Bx, x \rangle^{+\mu}},\; u(0)=0, $$

$\star$ PrTS: the stabilizing control is
$$ {\blue u(x) = -\frac{\omega_0}{\beta}  - \rho \langle Bx, x \rangle^{-\mu}- \rho\langle Bx, x \rangle^{+\mu}},\; u(0)=0, $$
with $u(0)=0$ and $0<\mu<1/2.$

\subsection{Application to  linear systems (FTS)}

We consider the linear system:
$$
\textcolor{blue}{(\text{LS})} \hspace{1cm} \dot{x}(t) = Ax(t) + \textcolor{red}{Lv(t)}, \quad x(0) = x_0 \in H.
$$

The above necessary conditions for FTS/FxTS/PrTS for  (BLS) apply to the linear system  (LS) with $L^*$ instead of  $B$.

We have the following  FTS result  for the linear system (LS), which follows from the bilinear case with $B=LL^*:$

\begin{corollaire}

\begin{itemize}

  \item Let $A$ generate a quasi-contractive semigroup $S(t)$ of type $\omega_0$, and assume that {\blue $ \ker(L^*)^\perp $ is invariant} under the semigroup $ S(t) $ and that:  $\exists \alpha >0$ s.t
      $${\blue \|L^*x\|\ge \alpha \|x\|, \forall x\in \ker(L^*)^\perp}.$$

    \vspace{0.25cm}

 \item Then  the system \textcolor{blue}{(LS)} is finite-time stable on \textcolor{blue}{$ \ker(L^*)^\perp $}, with the following  stabilizing control:
   $$
  {\blue v = - \frac{\omega_0}{\alpha^2} L^*x - \|L^*x\|^{-2\mu} L^*x, \; (0 < \mu < 1/2)}.
  $$

\end{itemize}

\end{corollaire}

\begin{itemize}

\item The  control defined by  $$
  {\blue v = - \frac{\omega_0}{\alpha^2} L^*x - \|L^*x\|^{-2\mu} L^*x- \|L^*x\|^{+2\mu} L^*x, \; (0 < \mu < 1/2)}.
  $$
leads to FxTS.\\
 Moreover if $\ker(L^*)^\perp$) is absorbent in finite time for $S(t)$, then we have the GFTS.

\item Furthermore, the  control defined by  $$
  {\blue v = - \frac{\omega_0}{\alpha^2} L^*x - \rho \|L^*x\|^{-2\mu} L^*x- \rho\|L^*x\|^{+2\mu} L^*x, \; (0 < \mu < 1/2)}.
  $$
leads to PrTS, where the gain control $\rho>0$ is chosen depending on the given prescribed time.

\end{itemize}

\subsection{FTS: Without invariance assumption}

\begin{center}
    In the sequel, we consider the Lyapunov function candidate:
    $ V(x) = \langle Bx, x \rangle.$
\end{center}

$\star$ We will give a FTS result on  $ \ker(B)^c\supset\ker(B)^\perp-(0)  $ for the system (BLS).

    $\star$ The (generalized) Lyapunov function  $ V=V(x) = \langle Bx, x \rangle $ is {\blue positive but not necessarily defined} on the whole space, so establishing FTS may require \textcolor{blue}{additional (observability) assumptions}.

 We formally compute,  along the trajectories of the system,
$$
{\blue \dot{V} = 2 \langle Ax, Bx \rangle + 2 u(t)\langle Bx, Bx \rangle}.
$$
 To ensure that $ V $ satisfies a differential inequality of the form ${\blue \dot{V}\le -m V^\alpha}, $ for some $ m>0, 0<\alpha<1$, we design  a control $u=u_1+u_2$, where  the controls  $u_1$ and $u_2,$ have the following roles:

\vspace{0.2cm}

\begin{itemize}
    \item[$\star$] $ u_1 $: will be chosen to 'cancel' the term $ \langle Ax, Bx \rangle $.

    \item[$\star$] $ u_2 $: will be  selected such that $    \dot{V}(x) \le -m V(x)^\alpha. $
\end{itemize}

\vspace{0.25cm}

$\bullet$ Let us consider the assumption
$$
({\cal AB}) \;\;\; {\blue \langle Ax,x\rangle_B \le \alpha \|x\|^2_B, \; \forall x \in D(A)},
$$
where $\langle \cdot, \cdot \rangle_B = \langle B \cdot, \cdot \rangle$ and $\|\cdot\|_B = \langle B \cdot, \cdot \rangle^{1/2}$.

 This assumption  means that the operator $A$ is quasi-dissipative with respect to $\langle \cdot, \cdot \rangle_B$.

\vspace{0.5cm}

$\star$ We can replace the above assumption by the existence of a function $ f: H \to \mathds{R} $ such that
$$
\frac{\langle Ax, x \rangle_B}{\|x\|^2_B} \le f(x), \; \forall x \in D(A)- \ker(B),
$$
with $ F := f \; B $ being a Lipschitz function.

\begin{theoreme}
 Let  assumptions $({\cal A})-({\cal B})-({\cal AB})$ hold, and assume that
for all $\xi \in H,$ we have
$$ {\blue
(OBS)\;\;\; BS(t)  \xi =0,\;\; \forall t\ge 0 \;\;\Rightarrow  \xi=0\cdot}
 $$
  Then the control law
$$
u(x) = -\alpha - \langle Bx,x\rangle^{-\mu} {\bf 1}_{(Bx\ne0)}, \; (0 < \mu < 1/2)
$$
stabilizes the system $(LS)$ in finite time for any $x_0$ s.t ${\blue B x_0\ne0}$ (in particular those of $(\ker B)^\perp$).
\\
 Furthermore, the settling time is such that
$$
T_* \le \frac{\langle Bx_0,x_0\rangle^{\mu}}{\beta \mu}.
$$
\end{theoreme}

\textbf{Proof:}

 The closed-loop system is
$$
\dot{x}=Ax-\alpha Bx-\langle Bx,x\rangle^{-\mu}Bx .
$$

By the same argument as above, it possesses a unique global solution $x$, and the function
$
V(t)=\langle Bx(t),x(t)\rangle
$
satisfies the Ineq:
$$
\dot{V}\le -2\beta V^{1-\mu/2},
$$
which is valid in a non trivial interval as $Bx_0\ne 0.$\\
It follows that
$$
\langle Bx(t),x(t)\rangle=0, \quad \forall t\ge
T_*:=\frac{\langle Bx_0,x_0\rangle^{\mu/2}}{2\beta \mu}.
$$
Hence,
$$
\dot{x}=Ax, \quad \forall t\ge T_*,
$$
which implies
$$
x(t)=S(t-T_*)x(T_*), \quad \forall t\ge T_* .
$$

Therefore,
$$
BS(t-T_*)x(T_*)=0, \quad \forall t\ge T_*,
$$
and consequently
$$
BS(t)x(T_*)=0, \quad \forall t\ge 0 .
$$
The observation condition (OBS) then yields $x(T_*)=0$,
which implies the finite-time stabilization (FTS).

\begin{remarque}
  Some extensions in the case without invariance assumption:

$\star$ The results on FxTS $\&$ PrTS
remain valid as in the invariance case.
\end{remarque}

\section{Applications }

\begin{exemple}
(FTS of a finite-dimension example.)
	
 Let $\mathcal{H}=\mathbb{R}^4$ and consider the following bilinear system:
	\begin{equation}\label{sydf}
		\left\{\begin{array}{l}
			\frac{d}{d t} x_1(t)= x_2(t)+2x_4(t) +u(x) x_1(t) \\
			\frac{d}{d t} x_2(t)=x_1(t)+u(x) x_2(t) \\
			\frac{d}{d t} x_3(t)= x_3(t) + u(x)x_3(t)\\
			\frac{d}{d t} x_4(t)=0
		\end{array}\right.
	\end{equation}
	
	$\bullet$ Here, we have: $x(t)=(x_1(t),x_2(t),x_3(t),x_4(t))^\intercal$ and
	$$
	A=\left(\begin{array}{cccc}
		0& 1 & 0 & 2 \\
		1& 0 & 0 & 0 \\
		0 & 0 & 1 &0\\
		0&0&0&0
	\end{array}\right),  \quad
	B=\left(\begin{array}{llll}
		1 & 0 & 0 & 0 \\
		0 & 1 & 0 &0\\
		0 & 0 & 1 & 0\\
		0 & 0 & 0 & 0
	\end{array}\right)
	$$
The matrix $A$ is dissipative and $B$ is a  self-adjoint positive matrix.
\\ We have $ \ker(B) = \operatorname{span}\{(0,0,0,1)\} $ and $ \ker(B)^{\perp} = \operatorname{span}\{(1,0,0,0), (0,1,0,0), (0,0,1,0)\}. $
\\ Here, the space $ \ker(B)^{\perp} $ is $ A $-invariant, then we achieve FTS of {\color{blue}(\ref{sydf})} for all $ x_0 \in \operatorname{span}\{(1,0,0,0), (0,1,0,0), (0,0,1,0)\} $, under the control law:
$$
	u(x) =
	\begin{cases}
	- \sqrt{2}	-\left(x_1(t)^2 + x_2(t)^2 + x_3(t)^2\right)^{-\mu}, & \text{if } (x_1, x_2, x_3) \neq (0,0,0) \\
		0, & \text{if } (x_1, x_2, x_3) = (0,0,0)
	\end{cases}
	\label{contr001}
$$
with $\omega_0=\sqrt{2}$ and  $ 0 < \mu < 1/2 $.\\
Moreover we have
$$
T_*(x_0)\le \frac{\|x_0\|^{2\mu}}{2\mu}
$$
 We have the  FxTS of {\color{blue}(\ref{sydf})} for all $ x_0 \in \operatorname{span}\{(1,0,0,0), (0,1,0,0), (0,0,1,0)\} $, under the control law:
$$
	u(x) =
		-\sqrt{2} -\left(x_1(t)^2 + x_2(t)^2 + x_3(t)^2\right)^{-\mu}-\left(x_1(t)^2 + x_2(t)^2 + x_3(t)^2\right)^{\mu},
$$
$ \text{if } (x_1, x_2, x_3) \neq (0,0,0),$ and
$$
u(x)=
		0, \; \text{if } (x_1, x_2, x_3) = (0,0,0)
	$$
with   $ 0 < \mu < 1/2 $.\\
Moreover we have
$$
T_*(x_0)\le \frac{\pi}{4\mu}
$$
The  PrTS of {\color{blue}(\ref{sydf})} occurs for all $ x_0 \in \operatorname{span}\{(1,0,0,0), (0,1,0,0), (0,0,1,0)\} $, under the control law
$$
	u(x) =
		-\sqrt{2} -\lambda\left(x_1(t)^2 + x_2(t)^2 + x_3(t)^2\right)^{-\mu}-\lambda\left(x_1(t)^2 + x_2(t)^2 + x_3(t)^2\right)^{\mu},
$$
$ \text{if } (x_1, x_2, x_3) \neq (0,0,0),$ and
$$
u(x)=
		0, \; \text{if } (x_1, x_2, x_3) = (0,0,0)
	$$
with   $ 0 < \mu < 1/2 $.\\
Here, the gain control $\lambda>0$ is chosen such that
$$
\frac{\pi}{4\lambda\mu} \le T
$$
That is
$$
\lambda\ge \frac{\pi}{4T\mu},
$$
where $T>0$ is the given prescribed time.

\end{exemple}

\begin{exemple} (FTS of a  Heat Equation).

Consider the linear heat equation:
\begin{equation*}
	\left\{
	\begin{array}{ll}
		\displaystyle\frac{\partial}{\partial t} y(x,t) =  y_{xx}(x,t) +  \sum_{1\le j \le q} v_j(t) a_j \phi_j, & (t, x) \in \mathbb{R}^{+} \times (0,1) \\
		y(0,t) = y(1,t) = 0, & t \in \mathbb{R}^{+} \\
		y( x,0) = y_0(x), & x \in (0,1)
	\end{array}
	\right.
\end{equation*}
with  $ a_j> 0,\; \forall 1\le j \le q$ for some $ q \in \mathbb{N}^*$.\\
 Where $\phi_j, \; j\ge 1$ are the orthonormal eigenfunctions of the system's operator: $A=\partial_{xx}$ with domain $D(A)=H_0^1(0,1)\cap H^2(0,1).$\\ Here, we have a linear system of the form (LS) with $L: U:=\mathbb{R}^q \to H$:
$$
L(x_1,..,x_q)=\sum_{1\le j \le q} x_j a_j \phi_j
$$
and with corresponding bilinear control operator $B=LL^*: L^2(0,1) \rightarrow L^2(0,1)$, given by
$$
B y = \sum_{j=1}^{q} a_j \langle y, \phi_j \rangle \phi_j,\; y\in L^2(0,1).
$$
Here, $\beta=\min_{1\le j\le q}(a_j).$
\\We have  the finite-time stability  for all initial states
$$
x_0 \in span\{ \phi_j,\; 1\le j\le q\},
$$
with a settling time s.t
$$T_*(x_0)\le \frac{\|x_0\|^{2\mu}}{2\mu \beta^{1-\mu}}.$$
The stabilizing control is given by
\begin{equation*}
	u(y) =
	\begin{cases}
		-\left( \displaystyle\sum_{1\le j \le q} a_j \langle y, \phi_j \rangle^2 \right)^{-\mu}, & \text{if } By \neq 0 \\
		0, & \text{if } By = 0
	\end{cases}
\end{equation*}
with $ 0 < \mu < 1/2 $.\\
Observing that $ \ker(B)=\overline{span}\{\phi_k,\; k>q \} $ is $ S(t) $-invariant, we conclude that finite-time stability (FTS) does not occur for $ x_0 \in \ker(B) $, regardless the choice of the control.
\\
We have  the FxTS  for all initial states
$$
x_0 \in span\{ \phi_j,\; 1\le j\le q\},
$$
and we have
$$T_*(x_0)\le \frac{\pi}{4\mu \beta^{1-\mu}}.$$
 The stabilizing control is given by
\begin{equation*}
	u(y) =
	\begin{cases}
		-\left( \displaystyle\sum_{1\le j \le q} a_j \langle y, \phi_j \rangle^2 \right)^{-\mu}-\left( \displaystyle\sum_{1\le j \le q} a_j \langle y, \phi_j \rangle^2 \right)^{\mu}, & \text{if } By \neq 0 \\
		0, & \text{if } By = 0
	\end{cases}
\end{equation*}
with $ 0 < \mu < 1/2 $.\\
We have  the PrTS  for all initial states
$$
x_0 \in span\{ \phi_j,\; 1\le j\le q\},
$$
and we have
$$T_*(x_0)\le \frac{\pi}{4\mu \beta^{1-\mu}}.$$
 The stabilizing control is given by
\begin{equation*}
	u(y) =
	\begin{cases}
		-\lambda \left( \displaystyle\sum_{1\le j \le q} a_j \langle y, \phi_j \rangle^2 \right)^{-\mu}-\lambda\left( \displaystyle\sum_{1\le j \le q} a_j \langle y, \phi_j \rangle^2 \right)^{\mu}, & \text{if } By \neq 0 \\
		0, & \text{if } By = 0
	\end{cases}
\end{equation*}
with $ 0 < \mu < 1/2 $ and  the gain control $\lambda>0$ is chosen such that
$$
\lambda\ge \frac{\pi}{4T\mu \beta^{1-\mu}},
$$
where $T>0$ is the given prescribed time.

\begin{remarque}

 Let us consider the situation where the bilinear control operator is given by
$$
B y = \sum_{j=p}^{+\infty} \frac{1}{j} \langle y, \phi_j \rangle \phi_j,
\qquad y \in L^2(0,1),
$$
(with $p \in \mathbb{N}^*$), which is not coercive on
$$
\ker(B)^\perp = \overline{\mathrm{span}}\{\phi_k,\; k > p\}.
$$
However, the coercivity holds on any finite-dimensional subspace
$$
F = \mathrm{span}\{\phi_k,\; p < k \le q\}, \qquad q > p,
$$
on which FTS and  FxTS  can be achieved according to the above results.
\end{remarque}
\end{exemple}

\begin{exemple}(FTS of a Transport Equation).

Let $\Omega=(0,+\infty)$ and $\omega=(a,+\infty),\; a>0$, and consider the following bilinear system:
 \begin{equation}\label{eq1}
\begin{cases}
z_{t}(x, t)=z_{x}(x, t)+\chi_{\omega} u(t) z(x, t), & (x, t) \in \Omega \times(0,+\infty), \\
z(x, 0)=z_{0}(x), & x \in \Omega,
\end{cases}
\end{equation}
where $u$ is a scalar bilinear control.
\\ Here, $A=\partial_x$ with domain $D(A)=\{y\in H^1(0,+\infty):\; y(0)=0\}$ generates a contraction semigroup (even isometric) on $H=L^2(0,+\infty),$ given by:
$$
\big(S(t)y\big)(x)=y(x-t){\bf 1}_{(x-t \ge0)},\; \forall y\in L^2(0,+\infty).
$$
Here, the necessary condition is satisfied and the space
$$
\ker(B)^\perp=\{y\in L^2(0,+\infty): \; y=0, \; a.e. \; x\in (0,a)\}
$$
is invariant under $S(t)$.
\\
 First, for all $y_{0} \in L^2(0,+\infty)$ s.t $y_0(x)=0,$ for a.e., $ 0<x<a$, the system \eqref{eq1}   is finite-time stable
 with the following feedback control:
$$
    u(t)=-\| y(t)\|^{-2\mu} \mathbf{1}_{\big( y(t) \neq 0\big) }, \; (0<\mu<1).
$$
 Moreover, we observe that $\ker(B)^\perp$ is absorbent for the semigroup $S(t)$, as for every $y_0\in H,$ we have $S(a)y_0=0, $ a.e $x\in (0,a)$.\\
Then we have the global FTS under the control
$$
u=\left\{
    \begin{array}{ll}
      0, & 0\le t \le a \\
            -\|  y(t)\|^{-2\mu} \mathbf{1}_{\big(  y(t) \neq 0\big)}, & t>a
    \end{array}
  \right.
$$
and we have
$$
T_{*}(x_0)\le a+ \frac{\|x_0\|^{2\mu}}{2\mu}
$$
 Moreover, the global FxTS occurs when using the control:
$$
u=\left\{
    \begin{array}{ll}
      0, & 0\le t \le a \\
\\
            -\big(\|  y(t)\|^{-2\mu} + \|  y(t)\|^{+2\mu}\big) \mathbf{1}_{\big(  y(t) \neq 0\big)}, & t>a
    \end{array}
  \right.
$$
and we have
$$
\sup_{x_0} T_{*}(x_0)\le a+ \frac{\pi}{4\mu}
$$

\end{exemple}

\newpage

\begin{center}
  {\bf Concluding Remarks  }
  \end{center}

$\bullet$  While the {\blue additive} FTS control is continuous near the equilibrium, the {\blue bilinear} FTS control is not. However, this discontinuity is compensated by the fact that the control enters the system multiplicatively (through the state).

$\bullet$  The blow-up phenomenon observed in bilinear control cannot, in general, be avoided, as we can see in the context of the first (elementary) example discussed above, where a positive exponent in the feedback law cannot guarantee the FTS.

$\bullet$  Blow-up can be avoided in certain situations of FTS where backward uniqueness fails to hold.  For example, one can consider the case of nilpotent semigroup  $ S(t) $, in which situation one can take $ u = 0 $ regardless of the choice of $ B $ (or $ B = 0 $ regardless of the choice of $ u $).

\vspace{0.5cm}

\begin{center}
  {\bf Some problems that still are open}
  \end{center}

\vspace{0.5cm}

$\bullet$ FTS of the heat equation with local control:

$$
y_{t} = y_{xx} + u(t)\chi_{\omega} y, \quad \omega \subset \Omega.
$$

For the linear version,
$$
y_{t} = y_{xx} + \mathbf{1}_{\omega} v(t),
$$
some partial results can be obtained if the global existence of the solution is assumed [Amaliki and Ouzahra, 2026].

$\bullet$ The second approach to FTS seems to be more suitable for the degenerate case (work in progress):\\

$$y_{t}=(a(x)y_{x})_x+u(t) \mathbf{1}_{\omega} y, \; \omega\subset \Omega,$$

and

$$y_{t}=a(x)y_{xx}+u(t) \mathbf{1}_{\omega} y, \; \omega\subset \Omega.$$

\vspace{0.5cm}

$\bullet$ FTS of  the wave equation:

$$y_{tt}=y_{xx}+u_1(t) y + u_2(t)y_t.$$

Using a time-varying feedback control $u_1$ and taking $u_2=0,$ it was shown (see [Jammazi et al. 2023]) that for any given $T>0$

$$
\lim_{t\to T^-} y(t)=0
$$


\newpage

\begin{thebibliography}{99}

\bibitem{ammari2020feedback} K. Ammari and M. Ouzahra, Feedback stabilization for a bilinear control system under weak observability inequalities, {\it Automatica}, \textbf{113} (2020), 108821.

\bibitem{ammari2021feedback} K. Ammari, S. El Alaoui, and M. Ouzahra, Feedback stabilization of linear and bilinear unbounded systems in Banach space, {\it Systems \& Control Letters}, \textbf{155} (2021), 104987.

\bibitem{ammari2000stabilization} K. Ammari and M. Tucsnak, Stabilization of Bernoulli--Euler beams by means of a pointwise feedback force, {\it SIAM Journal on Control and Optimization}, \textbf{39}(4) (2000), 1160-1181.

\bibitem{ouzahra2025finite} Y. Amaliki and M. Ouzahra, Finite-time stabilization of a class of unbounded parabolic bilinear systems in Hilbert space, {\it Journal of Mathematical Analysis and Applications}, \textbf{551}(2) (2025), 129720.

\bibitem{ouzahra2026} Y. Amaliki and M. Ouzahra, Finite-TimeStabilizationofLinear1Systems, {\it Springer Nature Switzerland AG,  Control Theory and Inverse Problems, Trends in Mathematics,} (2026). doi.org$/10.1007/978-3-031-96278-3_10.$

\bibitem{attouch2014variational} H. Attouch, G. Buttazzo, and G. Michaille, {\it Variational analysis in Sobolev and BV spaces: applications to PDEs and optimization}, Vol. 17, Siam, 2014.

\bibitem{ball1977strongly} J. M. Ball, Strongly continuous semigroups, weak solutions, and the variation of constants formula, {\it Proc. Amer. Math. Soc.}, \textbf{63} (1977), 370--373.

\bibitem{ball1978asymptotic} J. M. Ball, On the asymptotic behaviour of generalized processes, with applications to nonlinear evolution equations, {\it J. Differential Equations.}, \textbf{27} (1978), 224--265.

\bibitem{ball1979feedback} J. M. Ball and M. Slemrod, Feedback stabilization of distributed semilinear control systems, {\it Appl. Math. Opt.}, \textbf{5} (1979), 169--179.

\bibitem{barbu2018controllability} V. Barbu, {\it Controllability and stabilization of parabolic equations}, Birkh\"auser, 2018.

\bibitem{benzaza2023strong} A. Benzaza, L. Lafhim, and M. Ouzahra, Strong and weak-star stabilisation of bilinear systems on nonreflexive Banach spaces, {\it International Journal of Control}, \textbf{96}(2) (2023), 302-308.

\bibitem{benzaza2023robust} A. Benzaza, A. Brouri, and M. Ouzahra, Robust stabilization for affine control systems in Banach spaces, {\it Asian Journal of Control}, \textbf{25}(1) (2023), 497-508.

\bibitem{berrahmoune1999stabilization} L. Berrahmoune, Stabilization and decay estimate for distributed bilinear systems, {\it Systems and Control Letters.}, \textbf{36} (1999), 167--171.

\bibitem{berrahmoune2001remarks} L. Berrahmoune, Y. Elboukfaoui, and M. Erraoui, Remarks on the feedback stabilization of system affine in control, {\it European Journal of Control}, \textbf{7}(1) (2001), 17-28.

\bibitem{berrahmoune2010stabilization} L. Berrahmoune, Stabilization of unbounded bilinear control systems in Hilbert space, {\it Journal of mathematical analysis and applications}, \textbf{372}(2) (2010), 645-655.


\bibitem{bhat2000finite} S. P. Bhat and D. S. Bernstein, Finite-time stability of continuous autonomous systems, {\it SIAM Journal on Control and optimization}, \textbf{38}(3) (2000), 751-766.

\bibitem{bhat2005geometric} S. P. Bhat and D. S. Bernstein, Geometric homogeneity with applications to finite-time stability, {\it Mathematics of Control, Signals and Systems}, \textbf{17}(2) (2005), 101-127.

\bibitem{bounit1999feedback} H. Bounit and H. Hammouri, Feedback stabilization for a class of distributed semilinear control systems, {\it Nonlinear Anal.}, \textbf{37} (1999), 953--969.

\bibitem{bounit2003comments} H. Bounit, Comments on the feedback stabilization for bilinear control systems, {\it Applied Mathematics Letters}, \textbf{16} (2003), 847--851.

\bibitem{bounab1991weak} A. Bounabat and J. P. Gauthier, Weak stabilizability of infinite dimensional nonlinear systems, {\it Applied Mathematics Letters}, \textbf{4} (1991), 95--98.

    \bibitem{botsC6}
 M. W. Botsko, An elementary proof of Lebesgue's differentiation theorem. The American Mathematical Monthly, {\bf 110} (2003), 834--838.



\bibitem{boutoulout2014unbounded} A. Boutoulout, R. El Ayadi, and M. Ouzahra, An unbounded stabilization problem of bilinear systems, {\it Mathematics and Computers in Simulation}, \textbf{102} (2014), 39-50.

\bibitem{chen2000exponential} M. S. Chen and S. T. Tsao, Exponential stabilization of a class of unstable bilinear systems, {\it IEEE Transactions on Automatic Control}, \textbf{45}(5) (2000), 989-992.

\bibitem{corduneanu2008principles} C. Corduneanu, {\it Principles of differential and integral equations}, Vol. 295, American Mathematical Society, 2008.

\bibitem{coron2017null} J. M. Coron and H. M. Nguyen, Null controllability and finite time stabilization for the heat equations with variable coefficients in space in one dimension via backstepping approach, {\it Archive for Rational Mechanics and Analysis}, \textbf{225}(3) (2017), 993-1023.

\bibitem{dorato2006overview} P. Dorato, An overview of finite-time stability, {\it Current trends in nonlinear systems and control}, (2006), 185-194.

\bibitem{engel2000one} K. J. Engel and R. Nagel, {\it One-parameter semigroups for linear evolution equations}, Vol. 194, Springer, 2000.

\bibitem{erugin1951continuouation} N. Erugin, On the continuouation of solutions of differential equations (in Russian), {\it Prikl. Mat. Mekh.}, \textbf{17}(4) (1951).

\bibitem{espitia2019boundary} N. Espitia, A. Polyakov, D. Efimov, and W. Perruquetti, Boundary time-varying feedbacks for fixed-time stabilization of constant-parameter reaction-diffusion systems, {\it Automatica}, \textbf{103} (2019), 398-407.

\bibitem{fattorini1974time} H. O. Fattorini, The time-optimal control problem in Banach spaces, {\it Applied Mathematics and Optimization}, \textbf{1} (1974), 163--188.

\bibitem{fattorini2001survey} H. O. Fattorini, A survey of the time optimal problem and the norm optimal problem in infinite dimension, {\it Cubo Mat Educational}, \textbf{3} (2001), 147--169.

\bibitem{fenza2024weak} K. Fenza and M. Ouzahra, Weak and strong stabilisation of a class of bilinear systems with time varying delay, {\it Journal of Control and Decision}, (2024), 1-15.

\bibitem{gutman1981stabilizing} P.O. Gutman, Stabilizing controllers for bilinear systems, {\it IEEE Trans. Automat. Control}, \textbf{AC-26} (1981), 917-922.

\bibitem{jammazi2023small} C. Jammazi, M. Ouzahra, and M. Sogor\'e, Small-time extinction with decay estimate of bilinear systems on Hilbert space, {\it Journal of Nonlinear Science}, \textbf{33}(4) (2023), 54.

\bibitem{jurjevic1978controllability} V. Jurjevic and J.P. Quinn, Controllability and stability, {\it J. Differential Equations}, \textbf{28} (1978), 381-389.

\bibitem{kamenkov1953stability} G. Kamenkov, On stability of motion over a finite interval of time (in Russian), {\it Journal of Applied Math. and Mechanics (PMM)}, \textbf{17} (1953), 529-540.

\bibitem{kato1967nonlinear} T. Kato, Nonlinear semigroups and evolution equations, {\it J. Math. Soc. Japan.}, \textbf{19}(4) (1967), 508--520.



\bibitem{kato2012perturbation} T. Kato, {\it Perturbation theory for linear operators}, Vol. 132. Springer Science \& Business Media, 2012.

\bibitem{khapalov2010controllability} A. Y. Khapalov, {\it Controllability of partial differential equations governed by multiplicative controls}, Springer, 2010.

\bibitem{lasiecka1993uniform} I. Lasiecka and D. Tataru, Uniform boundary stabilization of semilinear wave equations with nonlinear boundary damping, {\it Differential and Integral Equations}, \textbf{6} (1993), 507--533.

\bibitem{lebedev1954problem} A. Lebedev, The problem of stability in a finite interval of time (in Russian), {\it Journal of Applied Math. and Mechanics (PMM)}, \textbf{18} (1954), 75-94.

\bibitem{liu2001exponential} K. Liu, Z. Liu, and B. Rao, Exponential stability of an abstract nondissipative linear system, {\it SIAM journal on control and optimization}, \textbf{40}(1) (2001), 149-165.

\bibitem{lourini2025strong} A. Lourini, M. El Azzouzi, and M. Laabissi, Strong and exponential stabilization of linear boundary control systems, {\it Mathematical Control and Related Fields}, \textbf{15} (2025), 528--547.

\bibitem{mohler1973bilinear} R. R. Mohler, {\it Bilinear Control Processes}, Academic Press New York, 1973.

\bibitem{mohler1980overview} R. R. Mohler and W. J. Kolodziej, An overview of bilinear system theory and applications, {\it IEEE Transactions on Systems, Man and Cybernetics}, \textbf{10} (1980), 683--688.

\bibitem{moulay2006finite} E. Moulay and W. Perruquetti, Finite time stability and stabilization of a class of continuous systems, {\it Journal of Mathematical analysis and applications}, \textbf{323}(2) (2006), 1430-1443.

\bibitem{muller1989strong} S. Muller, Strong Convergence and Arbitrarily Slow Decay of Energy for a Class of Bilinear Control Problems, {\it Journal of Diff. Eq.}, \textbf{81} (1989), 50-67.

\bibitem{najib2024stabilization} H. Najib and M. Ouzahra, Stabilization in Finite Time of a Class of Unbounded Non-linear Systems, {\it Journal of Dynamical and Control Systems}, \textbf{30}(2) (2024), 17.

\bibitem{ouzahra2008strong} M. Ouzahra, Strong stabilization with decay estimate of semilinear systems, {\it Systems \& control letters}, \textbf{57}(10) (2008), 813-815.

\bibitem{ouzahra2010exponential} M. Ouzahra, Exponential and weak stabilization of constrained bilinear systems, {\it SIAM Journal on Control and Optimization}, \textbf{48}(6) (2010), 3962-3974.

\bibitem{ouzahra2012stabilisation} M. Ouzahra, A. Tsouli, and A. Boutoulout, Stabilisation and polynomial decay estimate for distributed semilinear systems, {\it International Journal of Control}, \textbf{85} (2012), 451--456.

\bibitem{ouzahra2020uniform} M. Ouzahra, Uniform exponential stabilization of nonlinear systems in Banach spaces, {\it Mathematical Methods in the Applied Sciences}, \textbf{43}(12) (2020), 7311-7325.

\bibitem{ouzahra2021exponential} M. Ouzahra, Exponential stabilization of unstable bilinear systems in finite-and infinite-dimensional spaces, {\it IEEE Transactions on Automatic Control}, \textbf{66}(12) (2021), 5982-5989.

\bibitem{ouzahra2021finite} M. Ouzahra, Finite-time control for the bilinear heat equation, {\it European Journal of Control}, \textbf{57} (2021), 284-293.


\bibitem{pazy1983semi} A. Pazy, {\it Semi-groups of linear operators and applications to partial differential equations}, Springer Verlag, New York, 1983.

\bibitem{polyakov2011nonlinear} A. Polyakov, Nonlinear feedback design for fixed-time stabilization of linear control systems, {\it IEEE transactions on Automatic Control}, \textbf{57}(8) (2011), 2106-2110.

\bibitem{polyakov2018homogeneous} A. Polyakov, J. M. Coron, and L. Rosier, On homogeneous finite-time control for linear evolution equation in Hilbert space, {\it IEEE Transactions on Automatic Control}, \textbf{63}(9) (2018), 3143-3150.

\bibitem{polyakov2020generalized} A. Polyakov, {\it Generalized homogeneity in systems and control}, 2020.

\bibitem{preisC6}
D. Preiss, Differentiability of Lipschitz functions on Banach spaces,  Journal of Functional Analysis, {\bf 91} (1990), 312--345.


\bibitem{efimov2021finite} D. Efimov and A. Polyakov, Finite-time stability tools for control and estimation, 2021.

\bibitem{preiss1990differentiability} D. Preiss, Differentiability of Lipschitz functions on Banach spaces, {\it Journal of Functional Analysis}, \textbf{91} (1990), 312--345.

\bibitem{quinn1980stabilization} J.P. Quinn, Stabilization of bilinear systems by quadratic feedback control, {\it J. Math. Anal. Appl.}, \textbf{75} (1980), 66-80.

\bibitem{resendepereira2021nonlinear} H. W. Resende Pereira, R. Kawakami Harrop Galv$\check{a}$o, and T. G. Yoneyama, Nonlinear predictive control employing Carleman bilinearization and fixed search directions, {\it Asian Journal of Control}, \textbf{23} (2021), 2565--2574.

\bibitem{rhandi1998positivity} A. Rhandi, Positivity and Stability for a Population Equation with Diffusion on, {\it Positivity}, \textbf{2} (1998), 101--113.

\bibitem{rhandi1999asymptotic} A. Rhandi and H. Schnaubelt, Asymptotic behaviour of a non-autonomous population equation with diffusion in $L^1,$ {\it Discrete and Continuous Dynamical Systems}, \textbf{5} (1999), 663--684.

\bibitem{roxin1966finite} E. Roxin, On Finite Stability in Control Systems, {\it Rendiconti del Circolo Matematico di Palermo}, \textbf{15} (1966), 273-283.

\bibitem{ryan1983asymptotically} E.P. Ryan and N. Buckingham, On asymptotically stabilizing feedback control of bilinear systems, {\it IEEE Transactions on Automatic Control}, \textbf{28}(8) (1983), 863-864.

\bibitem{ryan1984optimal} E.P. Ryan, Optimal feedback control of bilinear systems, {\it Journal of Optimization Theory and Applications}, \textbf{44} (1984), 333-362.

\bibitem{slemrod1978stabilization} M. Slemrod, Stabilization of bilinear control systems with applications to nonconservative problems in elasticity, {\it SIAM. J. Control and optim.}, \textbf{16} (1978), 131-141.

  \bibitem{Vieru2005}  Vieru, A. (2005). On null controllability of linear systems in Banach spaces. Systems \& control letters, 54(4), 331-337.

\bibitem{zhang2019finite} C. Zhang, Finite-time internal stabilization of a linear 1-D transport equation, {\it Systems \& Control Letters}, \textbf{133} (2019), 10452.

\end{thebibliography}
\end{document}